\documentclass[12pt,oneside,english]{amsart}
\usepackage[T1]{fontenc}
\usepackage[latin9]{inputenc}
\usepackage{geometry}
\usepackage{enumerate}
\usepackage{amsthm}
\usepackage{amstext}
\usepackage{amssymb}
\usepackage{stmaryrd}
\usepackage{setspace}
\usepackage{esint}

\makeatletter

\numberwithin{equation}{section}
\numberwithin{figure}{section}
\theoremstyle{plain}
\newtheorem{thm}{\protect\theoremname}[section]
  \theoremstyle{plain}
  \newtheorem{lem}[thm]{\protect\lemmaname}
  \theoremstyle{plain}
  \newtheorem{prop}[thm]{\protect\propositionname}
  \theoremstyle{definition}
  \newtheorem{problem}[thm]{\protect\problemname}
  \theoremstyle{plain}
  \newtheorem{cor}[thm]{\protect\corollaryname}
  \theoremstyle{definition}
  \newtheorem{example}[thm]{\protect\examplename}
  \theoremstyle{remark}
  \newtheorem*{rem*}{\protect\remarkname}
  \theoremstyle{definition}
  \newtheorem{defn}[thm]{\protect\definitionname}

\makeatother

\usepackage{babel}
  \providecommand{\corollaryname}{Corollary}
  \providecommand{\definitionname}{Definition}
  \providecommand{\examplename}{Example}
  \providecommand{\lemmaname}{Lemma}
  \providecommand{\problemname}{Problem}
  \providecommand{\propositionname}{Proposition}
  \providecommand{\remarkname}{Remark}
\providecommand{\theoremname}{Theorem}

\begin{document}

\title{analytic aspects of the bi-free partial $R$-transform}

\author{hao-wei huang and jiun-chau wang}
\begin{abstract}
It is shown that the bi-freely infinitely divisible laws, and only
these, can be used to approximate the distributions of sums of
identically distributed bi-free pairs of random variables from commuting faces.
Furthermore, the necessary and sufficient conditions for this approximation
are found. Bi-free convolution semigroups of measures and their L\'{e}vy-Khintchine representations are also studied here from an infinitesimal point of view. The proofs relies on the harmonic analysis
machinery we developed for integral transforms of two variables, without reference to the combinatorics of moments and bi-free cumulants.\end{abstract}

\date{May 30, 2015; revised on April 28, 2016}

\address{Department of Mathematics and Statistics, Queen's University, Kingston, Ontario K7L 3N6, Canada}

\curraddr{Department of Applied Mathematics, National Sun Yat-sen University, No. 70, Lienhai Road, Kaohsiung 80424, Taiwan, R.O.C.}

\email{hwhuang@math.nsysu.edu.tw}

\address{Department of Mathematics and Statistics, University of Saskatchewan,
Saskatoon, Saskatchewan S7N 5E6, Canada}

\email{jcwang@math.usask.ca}



\maketitle

\section{introduction}
The purpose of this paper is to develop a harmonic analysis approach to the partial $R$-transform and infinitely divisible laws in Voiculescu's bi-free probability theory. 

Following \cite{V14,V15}, given a two-faced pair $(a,b)$ of left variable $a$ and right variable $b$ in a $C^{*}$-probability space $(\mathcal{A},\varphi)$, its \emph{bi-free partial $R$-transform} $R_{(a,b)}$ is defined as the generating series \[R_{(a,b)}(z,w)=\sum_{m,n\geq 0}R_{m,n}(a,b)z^{m}w^{n}\] of the ordered bi-free cumulants $\{R_{m,n}(a,b):m,n \geq 0\}$ for the pair $(a,b)$. As shown by Voiculescu, this partial $R$-transform actually converges absolutely to the following holomorphic function near the point $(0,0)$ in $\mathbb{C}^2$: \begin{equation}\label{eq:1.1} R_{(a,b)}(z,w)=1+zR_{a}(z)+wR_{b}(w)-zw/G_{(a,b)}(1/z+R_{a}(z),1/w+R_{b}(w)),\end{equation} where $R_{a}$ and $R_{b}$ are respectively the usual $R$-transforms of $a$ and $b$, and the function $G_{(a,b)}$ is given by \[G_{(a,b)}(z,w)=\varphi((zI-a)^{-1}(wI-b)^{-1}).\] Moreover, if two two-faced pairs $(a_1,b_1)$ and $(a_2,b_2)$ are \emph{bi-free} as in \cite{V14}, then one has \[R_{(a_1+a_2,b_1+b_2)}(z,w)=R_{(a_1,b_1)}(z,w)+R_{(a_2,b_2)}(z,w)\] for $(z,w)$ near $(0,0)$. 

Since their introduction in 2013, the bi-free $R$-transform and bi-free cumulants have been the subject of several investigations \cite{CNS, GHM, MN, Skoufranis14} from the combinatorial perspective. (We also refer the reader to the original papers \cite{V14, V15, V15ST, V15EXT} for the basics of bi-free probability and to \cite{CNS2, FW, Skoufranis15} for other developments of this theory.) Here in this paper, we would like to contribute to the study of the bi-free $R$-transform by initiating a new direction which is solely based on the harmonic analysis of integral transforms in two variables. Of course, to accommodate objects like measures or integral transforms, we naturally confine ourselves into the case where all left variables \emph{commute} with all right variables. Thus, the distribution for a two-faced pair of commuting selfadjoint variables is the composition of the expectation functional with the joint spectral measure of these variables, which is a compactly supported Borel probability measure on $\mathbb{R}^2$. In particular, the map $G_{(a,b)}$ now becomes the Cauchy transform of the distribution of the pair $(a,b)$. Furthermore, according to the results in \cite{V14}, given two compactly supported probabilities $\mu_1$ and $\mu_2$ on $\mathbb{R}^2$, one can find two bi-free pairs $(a_1,b_1)$ and $(a_2,b_2)$ of commuting left and right variables such that the law of $(a_j,b_j)$ is the measure $\mu_j$ ($j=1,2$) and the \emph{bi-free convolution} $\mu_1\boxplus\boxplus \mu_2$ of these measures is the distribution of the sum $(a_1,b_1)+(a_2,b_2)=(a_1+a_2,b_1+b_2)$. To reiterate, we now have $R_{\mu_1\boxplus\boxplus \mu_2}=R_{\mu_1}+R_{\mu_2}$ near the point $(0,0)$ in this case. 
  
Under such a framework, we are able to develop a satisfactory theory for bi-free harmonic analysis of probability measures on the plane, and we show that the classical limit theory for infinitely divisible laws, due to L\'{e}vy and Khintchine, has a perfect bi-free analogue. 

The organization and the description of the results in this paper are as follows. We first begin with continuity results for the two-dimensional Cauchy transform in Section 2. Then we take \eqref{eq:1.1} as the new definition for the bi-free $R$-transform of a planar measure and prove similar continuity results for this transform. In Section 3, we are set to investigate the convergence properties of the scaled bi-free $R$-transforms $f_n=k_nR_{\mu_n}$, where $k_n\in\mathbb{N}$ and $\mu_n$ is a probability law on $\mathbb{R}^2$. We find the necessary and sufficient conditions for the pointwise convergence of $\{f_n\}_{n=1}^{\infty}$, and show that the pointwise limit $f=\lim_{n\rightarrow\infty}f_n$ will be a bi-free $R$-transform for some probability law $\nu$ if the limit $f$ should exist in a certain domain of $\mathbb{C}^2$. The class $\mathcal{BID}$ of \emph{bi-freely infinitely divisible laws} is then introduced as the family of all such limit laws $\nu$. Examples are provided and include bi-free analogues of Gaussian and Poisson laws. Other properties of this class such as compound Poisson approximation and a convolution semigroup embedding property are also studied here in Section 3. When applying our results to bi-free convolution of compactly supported measures, we obtain the criteria for the weak convergence of the measures $\mu_{n}\boxplus \boxplus \mu_n \boxplus \boxplus\cdots \boxplus \boxplus\mu_n$ ($k_n$ times) and the characterization of their limit (namely, being infinitely divisible). Interestingly enough, our limit theorems do not depend on whether the function $f_n$ is a bi-free $R$-transform or not; that is, the existence of the bi-free convolution for measures with unbounded support does not play a role here. We do, however, show that the binary operation $\boxplus \boxplus$ can be extended from compactly supported measures to the class $\mathcal{BID}$. Finally, in Section 4, the bi-free $R$-transform $R_{\nu}$ of any law $\nu\in \mathcal{BID}$ is studied from a dynamical point of view. We show that $R_{\nu}$ arises as the time derivative of the Cauchy transforms corresponding to the bi-free convolution semigroup $\{\nu_t\}_{t\geq 0}$ generated by the law $\nu$. We then obtain a canonical integral representation for $R_{\nu}$, called the \emph{bi-free L\'{e}vy-Khintchine formula}, from this aspect of the bi-free $R$-transform. 

\subsection*{Acknowledgements}
The second named author would like to express his gratitude to Professor Joel Zinn for stimulating conversations that led him to Theorem 3.6 and Proposition 3.11 in this paper. The first-named author was supported through a Coleman Postdoctoral Fellowship at Queen's University and a startup grant for new faculty at National Sun Yat-sen University, and the second-named author was supported by the NSERC Canada Discovery Grant RGPIN-402601. Both authors would like to thank the anonymous referee for his/her review and helpful comments.

\section{continuity theorems}

\subsection{Two-dimensional Cauchy transforms }

We start with some continuity results for the Cauchy transform of
two variables. These results are not new and must be known already by harmonic analysts. Since we did not find an appropriate reference for them, we provide their proofs here for the sake of completeness. 

Denote by $\mathbb{C}^{+}=\left\{ z\in\mathbb{C}:\Im z>0\right\} $
the complex upper half-plane and by $\mathbb{C}^{-}$ the lower one.
For a (positive) planar Borel measure $\mu$ satisfying the growth
condition
\begin{equation}
\int_{\mathbb{R}^{2}}\frac{1}{\sqrt{1+s^{2}}\sqrt{1+t^{2}}}\, d\mu(s,t)<\infty,\label{eq: 2.1}
\end{equation}
the domain of definition for its \emph{Cauchy transform}
\[
G_{\mu}(z,w)=\int_{\mathbb{R}^{2}}\frac{1}{(z-s)(w-t)}\, d\mu(s,t)
\]
is the set $(\mathbb{C}\setminus\mathbb{R})^{2}=\{(z,w)\in\mathbb{C}^{2}:z,w\notin\mathbb{R}\}$
consisting of four connected components: $\mathbb{C}^{+}\times\mathbb{C}^{+}$,
$\mathbb{C}^{-}\times\mathbb{C}^{+}$, $\mathbb{C}^{-}\times\mathbb{C}^{-}$,
and $\mathbb{C}^{+}\times\mathbb{C}^{-}$. The function $G_{\mu}$
is holomorphic and satisfies the symmetry
\[
G_{\mu}(z,w)=\overline{G_{\mu}(\overline{z},\overline{w})},\qquad(z,w)\in(\mathbb{C}\setminus\mathbb{R})^{2}.
\]

Assume in addition that the Borel measure $\mu$ is finite on all compact subsets of $\mathbb{R}^2$, so that $\mu$ is a $\sigma$-finite Radon measure. Since the kernels \[\frac{1}{\pi^2}\frac{y^2}{(x^2+y^2)(u^2+y^2)},\quad y>0,\] form an approximate identity in the space $L^{1}(\mathbb{R}^2)$ with respect to the Lebesgue measure $dxdu$ on $\mathbb{R}^2$, a standard truncation argument and Fubini's theorem imply that for any compactly supported continuous function $\varphi$ on $\mathbb{R}^2$, one has the following inversion formula which recovers the measure $\mu$ as a positive linear functional acting on such $\varphi$: \[\int_{\mathbb{R}^2}\varphi\,d\mu = \lim_{y\rightarrow 0^{+}}\frac{1}{\pi^{2}}\int_{\mathbb{R}^2}\varphi(x,u)\left[\underbrace{\int_{\mathbb{R}^{2}}\frac{y}{(x-s)^{2}+y^{2}}\,\frac{y}{(u-t)^{2}+y^{2}}\, d\mu(s,t)}_{=f(x,u,y)}\right]\, dxdu.\]
Hence the Cauchy transform $G_{\mu}$ determines the underlying measure
$\mu$ uniquely. Indeed, take the imaginary part, we have
\[
\Im\left[\frac{G_{\mu}(x+iy,u+iy)-G_{\mu}(x+iy,u-iy)}{2i}\right]=f(x,u,y).
\]
Apparently, the definition of $G_{\mu}$ and the above properties
can be extended to any Borel signed measure $\mu$ whose total variation $|\mu|$
satisfies the growth condition (\ref{eq: 2.1}) and $|\mu|(K)<\infty$ for all compact $K\subset \mathbb{R}^2$.

Let $\pi_{1}$ and $\pi_{2}$ be the projections defined by $\pi_{1}(s,t)=s$
and $\pi_{2}(s,t)=t$ for $(s,t)\in\mathbb{R}^{2}$. For a Borel measure
$\mu$ on $\mathbb{R}^{2}$, its (principal) \emph{marginal} \emph{laws}
$\mu^{(j)}$ ($j=1,2$) are defined as $\mu^{(j)}=\mu\circ\pi_{j}^{-1}$,
the push-forward of $\mu$ by these projections. Denoting $\alpha_{z}=\sqrt{1+(\Re z/\Im z)^{2}}$
for any complex number $z\notin\mathbb{R}$, we say that $z\rightarrow\infty$
\emph{non-tangentially} (and write $z\rightarrow_{\leftslice}\infty$
to indicate this) if $\left|z\right|\rightarrow\infty$ and the quantity
$\alpha_{z}$ remains bounded. The notation $z,w\rightarrow_{\leftslice}\infty$ means that both $z$ and $w$ tend to infinity non-tangentially. Our first result says that the one-dimensional
Cauchy transform
\[
G_{\mu^{(j)}}(z)=\int_{\mathbb{R}}\frac{1}{z-x}\, d\mu^{(j)}(x),\qquad z\notin \mathbb{R},
\]
of the marginal law $\mu^{(j)}$ can be recovered from $G_{\mu}$ as a non-tangential
limit. Recall that a family $\mathcal{F}$ of finite
Borel signed measures on $\mathbb{R}^{2}$ is said to be \emph{tight}
if
\[
\lim_{m\rightarrow\infty}\sup_{\mu\in\mathcal{F}}\left|\mu\right|(\mathbb{R}^{2}\setminus K_{m})=0,
\]
where $K_{m}=\left\{ (s,t):\left|s\right|\leq m,\left|t\right|\leq m\right\}$. The tightness for measures supported on $\mathbb{R}$ is defined
analogously. Since \[\mathbb{R}^{2}\setminus K_{m}=\left\{(s,t):\left|s\right|> m,t\in \mathbb{R}\right\}\cup \left\{(s,t):s\in \mathbb{R},\left|t\right|> m\right\},\] the finite subadditivity of total variation measure shows that a family $\mathcal{F}$ of Borel signed measures
on $\mathbb{R}^{2}$ is tight if and only if the collection $\left\{ |\mu|^{(1)},|\mu|^{(2)}:\mu\in\mathcal{F}\right\} $
of the marginal laws forms a tight family of Borel signed measures on
$\mathbb{R}$.
\begin{lem}
Let $\mathcal{F}$ be a tight family of probability measures on $\mathbb{R}^{2}$.
Then for each $(z,w)\in(\mathbb{C}\setminus\mathbb{R})^{2}$, the
limits
\[
\begin{cases}
\lim_{\lambda\rightarrow_{\leftslice}\infty}\lambda G_{\mu}(z,\lambda)=G_{\mu^{(1)}}(z)\\
\lim_{\lambda\rightarrow_{\leftslice}\infty}\lambda G_{\mu}(\lambda,w)=G_{\mu^{(2)}}(w)
\end{cases}
\]
hold uniformly for $\mu\in\mathcal{F}$. Moreover, these two limits
are also uniform for $(z,w)$ in the union $\left\{ (z,w):\left|\Im z\right|\geq\varepsilon>0,\left|\Im w\right|\geq\delta>0\right\} $
of polyhalfplanes.\end{lem}
\begin{proof}
Observe that for $z,\lambda\notin\mathbb{R}$, $m>0$, and $\mu\in\mathcal{F}$,
we have
\begin{eqnarray*}
\left|\lambda G_{\mu}(z,\lambda)-G_{\mu^{(1)}}(z)\right| & = & \left|\int_{\mathbb{R}^{2}}\frac{1}{z-s}\left[\frac{\lambda}{\lambda-t}-1\right]\, d\mu(s,t)\right|\\
 & \leq & \frac{1}{\left|\Im z\right|}\int_{\left\{ (s,t):\left|t\right|\leq m\right\} }\left|\frac{t}{\lambda-t}\right|\, d\mu(s,t)\\
 &  & +\frac{1}{\left|\Im z\right|}\int_{\left\{ (s,t):\left|t\right|>m\right\} }\left|\frac{\lambda}{\lambda-t}-1\right|\, d\mu(s,t)\\
 & \leq & \frac{m}{\left|\Im z\right|\left|\Im\lambda\right|}+\frac{(\alpha_{\lambda}+1)}{\left|\Im z\right|}\mu(\mathbb{R}^{2}\setminus K_{m}).
\end{eqnarray*}
Likewise, for $\lambda,w\notin\mathbb{R}$ and $m>0$, we have
\[
\left|\lambda G_{\mu}(\lambda,w)-G_{\mu^{(2)}}(w)\right|\leq\frac{m}{\left|\Im\lambda\right|\left|\Im w\right|}+\frac{(1+\alpha_{\lambda})}{\left|\Im w\right|}\sup_{\mu\in\mathcal{F}}\mu(\mathbb{R}^{2}\setminus K_{m}).
\]
The result follows from these estimates.
\end{proof}
Since $\lim_{z\rightarrow_{\leftslice}\infty}zG_{\mu^{(1)}}(z)=1$
uniformly for $\mu^{(1)}$ in any tight family of probability measures
on $\mathbb{R}$ \cite{BerVoicu93}, we deduce from Lemma 2.1 that
\begin{equation}
G_{\mu}(z,w)=\frac{1}{zw}(1+o(1))\quad\text{as}\quad z,w\rightarrow_{\leftslice}\infty,\quad(z,w)\in(\mathbb{C}\setminus\mathbb{R})^{2},\label{eq:2.2}
\end{equation}
uniformly for $\mu$ within any tight family of probabilities on $\mathbb{R}^{2}$.
This non-tangential limiting behavior plays a role in our next result.

Recall that the set of all finite Borel signed measures on $\mathbb{R}^{2}$ is
equipped with the topology of weak convergence from duality with continuous and bounded functions on $\mathbb{R}^{2}$
under the sup norm. In this topology, a family of signed measures
is relatively compact if and only if it is tight and uniformly bounded
in total variation norms. Likewise, weak convergence of measures on
$\mathbb{R}$ is based on the duality with bounded continuous functions
on $\mathbb{R}$. By Prokhorov's theorem, a tight sequence of probability
measures contains a subsequence which converges
weakly to a probability measure. We write $\mu_{n}\Rightarrow\mu$
if the sequence $\{\mu_{n}\}_{n=1}^{\infty}$ converges weakly to
$\mu$ as $n\rightarrow\infty$. 
\begin{prop}
Let $\{\mu_{n}\}_{n=1}^{\infty}$ be a sequence of probability measures on
$\mathbb{R}^{2}$. Then the sequence $\mu_{n}$ converges weakly to
a probability measure on $\mathbb{R}^{2}$ if and only if \emph{(i)}
there exist two open subsets $U\subset\mathbb{C}^{+}\times\mathbb{C}^{+}$
and $V\subset\mathbb{C}^{+}\times\mathbb{C}^{-}$ such that the pointwise
limit $\lim_{n\rightarrow\infty}G_{\mu_{n}}(z,w)=G(z,w)$ exists for
every $(z,w)\in U\cup V$ and \emph{(ii)} $zwG_{\mu_{n}}(z,w)\rightarrow1$
uniformly in $n$ as $z,w\rightarrow_{\leftslice}\infty$.
Moreover, if $\mu_{n}\Rightarrow\mu$, then we have $G=G_{\mu}$.\end{prop}
\begin{proof}
Assume $\mu_{n}\Rightarrow\mu$ for some probability $\mu$ on $\mathbb{R}^{2}$.
The pointwise convergence $G_{\mu_{n}}\rightarrow G_{\mu}$ follows
from the definition of weak convergence and the estimate
\[
\frac{1}{\left|z-s\right|\left|w-t\right|}\leq\frac{1}{\left|\Im z\right|\left|\Im w\right|},\qquad s,t\in\mathbb{R};\;\Im z,\Im w\neq0.
\]
As we have seen earlier, the non-tangential limit (ii) is a consequence
of the tightness of $\{\mu_{n}\}_{n=1}^{\infty}$.

Conversely, we assume (i) and (ii). The uniform condition (ii) implies that
for any given $\varepsilon>0$, there corresponds $m=m(\varepsilon)>0$
such that
\[
\left|(iy)(iv)G_{\mu_{n}}(iy,iv)-1\right|<\varepsilon,\qquad y,v\geq m,\quad n\geq1.
\]
Taking $v\rightarrow\infty$ and fix $y=m$, Lemma 2.1 shows that
\begin{eqnarray*}
\varepsilon & \geq & \left|(im)G_{\mu_{n}^{(1)}}(im)-1\right|\\
 & \geq & -\Re\left[(im)G_{\mu_{n}^{(1)}}(im)-1\right]\\
 & = & \int_{\mathbb{R}}\frac{s^{2}}{m^{2}+s^{2}}\, d\mu_{n}^{(1)}(s)=\int_{\mathbb{R}^{2}}\frac{s^{2}}{m^{2}+s^{2}}\, d\mu_{n}(s,t)\geq\frac{1}{2}\mu_{n}(\left\{ (s,t):\left|s\right|>m\right\} )
\end{eqnarray*}
for every $n\geq1$. Similarly, we get
\[
\sup_{n\geq1}\mu_{n}(\left\{ (s,t):\left|t\right|>m\right\} )\leq2\varepsilon
\]
after taking $y\rightarrow\infty$ and fixing $v=m$. We conclude from these uniform estimates that the sequence $\{\mu_{n}\}_{n=1}^{\infty}$ is tight
and hence it possesses weak limit points, at least one of which is
a probability law on $\mathbb{R}^{2}$.

We shall argue that there can only be one weak limit for $\{\mu_{n}\}_{n=1}^{\infty}$
and therefore the entire sequence $\mu_{n}$ must converge weakly
to that unique probability limit law. Indeed, suppose $\mu$ and $\nu$
are both weak limits of $\{\mu_{n}\}_{n=1}^{\infty}$, then the condition
(i) implies that
\[
G_{\mu}(z,w)=G(z,w)=G_{\nu}(z,w),\qquad(z,w)\in U\cup V.
\]
Since $G_{\mu}$ and $G_{\nu}$ are holomorphic in $(\mathbb{C}^{+}\times\mathbb{C}^{+})\cup(\mathbb{C}^{+}\times\mathbb{C}^{-})$
and $U$ and $V$ are open sets, the Identity Theorem in multidimensional
complex analysis implies that $G_{\mu}=G_{\nu}$ in $(\mathbb{C}^{+}\times\mathbb{C}^{+})\cup(\mathbb{C}^{+}\times\mathbb{C}^{-})$.
Moreover, we can extend the functional equation $G_{\mu}=G_{\nu}$
to the whole $(\mathbb{C}\setminus\mathbb{R})^{2}$ by taking its
complex conjugation and conclude that $\mu$ and $\nu$ have the same
Cauchy transform. Since Cauchy transform determines the underlying
measure uniquely, we conclude that $\mu=\nu$, finishing the proof.
\end{proof}
Note that the condition (ii) in this proposition is in fact equivalent
to the tightness of the sequence $\left\{ \mu_{n}\right\} _{n=1}^{\infty}$.

We also have the following continuity result if the limit law has
been specified in advance; in which case, the pointwise convergence
of Cauchy transforms suffices for the weak convergence of measures.
\begin{prop}
Let $\mu$, $\mu_{1}$, $\mu_{2}$, $\cdots$ be probability measures
on $\mathbb{R}^{2}$. Then $\mu_{n}\Rightarrow\mu$ if and only if
$\lim_{n\rightarrow\infty}G_{\mu_{n}}(z,w)=G_{\mu}(z,w)$ for every
$(z,w)\in(\mathbb{C}\setminus\mathbb{R})^{2}$. \end{prop}
\begin{proof}
Only the ``if'' part needs a proof. Assume the pointwise convergence
$G_{\mu_{n}}\rightarrow G_{\mu}$. Since sets bounded in total variation norm are also weak-star
pre-compact sets, the sequence $\{\mu_{n}\}_{n=1}^{\infty}$ has a weak-star
limit point, say, $\sigma$. Observe that for every $(z,w)\in(\mathbb{C}\setminus\mathbb{R})^{2}$,
the corresponding Cauchy kernel
\[
\frac{1}{(z-s)(w-t)},\qquad(s,t)\in\mathbb{R}^{2},
\]
is a continuous function vanishing at infinity. Therefore, being a weak-star
limit point of the sequence $\{\mu_{n}\}_{n=1}^{\infty}$, the measure $\sigma$ must satisfy $G_{\sigma}=G_{\mu}$ on the domain $(\mathbb{C}\setminus\mathbb{R})^{2}$. We conclude that any weak-star limit $\sigma$ is in fact equal
to the given probability measure $\mu$ and therefore $\mu_{n}\Rightarrow\mu$ holds.
\end{proof}

\subsection{$R$-transforms}

We now turn to $R$-transform. Recall that a (truncated) Stolz angle
$\Delta_{\alpha,\beta}\subset\mathbb{C}^{-}$ at zero is the convex
domain defined by
\[
\Delta_{\alpha,\beta}=\left\{ x+iy\in\mathbb{C}^{-}:\left|x\right|<-\alpha y,\: y>-\beta\right\} ,
\]
where $\alpha,\beta>0$ are two parameters controlling the size of $\Delta_{\alpha,\beta}$.
The notation $\overline{\Delta_{\alpha,\beta}}$ means the reflection
$\left\{ \overline{z}:z\in\Delta_{\alpha,\beta}\right\} $. As shown
in \cite{BerVoicu93}, Stolz angles and their reflections are the natural
domains of definition for the usual one-dimensional $R$-transforms.

Given a Stolz angle $\Delta_{\alpha,\beta}$, we introduce the product
domain
\[
\Omega_{\alpha,\beta}=(\Delta_{\alpha,\beta}\cup\overline{\Delta_{\alpha,\beta}})\times(\Delta_{\alpha,\beta}\cup\overline{\Delta_{\alpha,\beta}})=\{(z,w):z,w\in \Delta_{\alpha,\beta}\cup\overline{\Delta_{\alpha,\beta}}\}.
\]
Since $z\rightarrow0$ in $\Delta_{\alpha,\beta}\cup\overline{\Delta_{\alpha,\beta}}$
if and only if $1/z\rightarrow_{\leftslice}\infty$, we have  $(z,w)\rightarrow(0,0)$ within $\Omega_{\alpha,\beta}$ if and only if $1/z,1/w\rightarrow_{\leftslice}\infty$. For notational
convenience, we will often write $\Delta$ for $\Delta_{\alpha,\beta}$
and $\Omega$ for $\Omega_{\alpha,\beta}$ in the sequel.

Following Voiculescu \cite{V14, V15}, given a probability $\mu$ on $\mathbb{R}^{2}$, its \emph{bi-free
partial} \emph{$R$-transform} (or, just $R$-transform for short)
is defined as
\begin{equation}\label{eq:2.3}
R_{\mu}(z,w)=zR_{\mu^{(1)}}(z)+wR_{\mu^{(2)}}(w)+\left[1-\frac{1}{h_{\mu}(z,w)}\right],
\end{equation}
where
\[
h_{\mu}(z,w)=G_{\mu}\left(1/z+R_{\mu^{(1)}}(z),1/w+R_{\mu^{(2)}}(w)\right)/zw
\]
and the function $R_{\mu^{(j)}}$ is the one-dimensional $R$-transform
for the marginal $\mu^{(j)}$. According to the non-tangential asymptotics
(\ref{eq:2.2}) and the fact that $1/\lambda+R_{\mu^{(j)}}(\lambda)=(1/\lambda)(1+o(1))\rightarrow_{\leftslice}\infty$
as $\lambda\rightarrow0$ within any Stolz angle at zero \cite{BerVoicu93}, there exists a small Stolz angle $\Delta$ such that the map $h_{\mu}$ is well-defined on the corresponding product domain
$\Omega=(\Delta\cup\overline{\Delta})\times (\Delta\cup\overline{\Delta})$, and the function $h_{\mu}$ never vanishes on $\Omega$. Therefore, the resulting bi-free $R$-transform $R_{\mu}$ is well-defined
and holomorphic on the set $\Omega$.

It is understood that we will always take such a set $\Omega$ as the domain of definition for the $R$-transform, unless the measure $\mu$ is compactly supported. Indeed, if $\mu$ has a bounded support, then the domain $\Omega$
can be chosen as an open bidisk centered at $(0,0)$, on which the
map $R_{\mu}$ admits an absolutely convergent power series expansion with real coefficients. In other words, the bi-free $R$-transform in this case extends analytically to a neighborhood of $(0,0)$. Also, the $R$-transform linearizes the
bi-free convolution of compactly supported measures, as shown in Voiculescu's work \cite{V14, V15}. In particular, in this case the sum of two such $R$-transforms
is another $R$-transform on their common domain of definition.

Finally, since the maps $R_{\mu^{(j)}}$ ($j=1,2$)
satisfy the symmetry property $R_{\mu^{(j)}(\lambda)}=\overline{R_{\mu^{(j)}}(\overline{\lambda})}$
\cite{BerVoicu93}, we also have
\[
R_{\mu}(z,w)=\overline{R_{\mu}(\overline{z},\overline{w})},\qquad(z,w)\in\Omega.
\]

As in the case of Cauchy transform, the one-dimensional $R$-transform
can be recovered from $R_{\mu}$ as a limit.
\begin{lem}
Let $R_{\mu}:\Omega\rightarrow\mathbb{C}$ be the $R$-transform of
a probability measure $\mu$ on $\mathbb{R}^{2}$. Then:
\begin{enumerate}[$\qquad(1)$]
\item For any $(z,w)\in\Omega$, we have
\[
\begin{cases}
\lim_{\lambda\rightarrow0}R_{\mu}(z,\lambda)=zR_{\mu^{(1)}}(z);\\
\lim_{\lambda\rightarrow0}R_{\mu}(\lambda,w)=wR_{\mu^{(2)}}(w).
\end{cases}
\]

\item $\lim_{(z,w)\rightarrow(0,0)}R_{\mu}(z,w)=0$.
\end{enumerate}
\end{lem}
\begin{proof}
We will only prove the limit $\lim_{\lambda\rightarrow0}R_{\mu}(z,\lambda)=zR_{\mu^{(1)}}(z)$,
for the second limit follows by the same argument, and (2) is a direct
consequence of the non-tangential limit (\ref{eq:2.2}).

Since $\lim_{\lambda\rightarrow0}\lambda R_{\mu^{(2)}}(\lambda)=0$
\cite{BerVoicu93}, it suffices to show that
\[
h_{\mu}(z,\lambda)=\frac{G_{\mu}\left(1/z+R_{\mu^{(1)}}(z),1/\lambda+R_{\mu^{(2)}}(\lambda)\right)}{z\lambda}\rightarrow1
\]
for any $z\in\Delta\cup\overline{\Delta}$ as $\lambda\rightarrow0$.
This, however, follows from Lemma 2.1, because $1/\lambda+R_{\mu^{(2)}}(\lambda)=(1/\lambda)(1+o(1))\rightarrow_{\leftslice}\infty$
as $\lambda\rightarrow0$, $\lambda\in\Delta\cup\overline{\Delta}$.
\end{proof}
We remark that if we take the tightness of measures into account,
then the limits in Lemma 2.4 can also be made uniform over any tight
family of probability measures on $\mathbb{R}^{2}$.

An immediate consequence of Lemma 2.4 is that the function $R_{\mu}$
determines the measure $\mu$ uniquely.
\begin{prop}
If two probability measures $\mu$ and $\nu$ have the same $R$-transform,
then $\mu=\nu$.\end{prop}
\begin{proof}
If $R_{\mu}=R_{\nu}$ on a domain $\Omega=(\Delta\cup\overline{\Delta})\times (\Delta\cup\overline{\Delta})$, then $\mu^{(j)}=\nu^{(j)}$
($j=1,2$) by Lemma 2.4. The definition \eqref{eq:2.3} implies further that \[G_{\mu}\left(1/z+R_{\mu^{(1)}}(z),1/w+R_{\mu^{(2)}}(w)\right)=G_{\nu}\left(1/z+R_{\nu^{(1)}}(z),1/w+R_{\nu^{(2)}}(w)\right)\] for $(z,w)\in\Omega$.

Because the image of the Stolz angle $\Delta$ under the map $\lambda \mapsto 1/\lambda+R_{\mu^{(j)}}(\lambda)$ contains a truncated cone \[\Gamma=\left\{ x+iy\in\mathbb{C}:\left|x\right|<ay,y>b\right\}\]
for some $a,b>0$ (cf. \cite{BerVoicu93}), we conclude that $G_{\mu}=G_{\nu}$ on the open set $(\Gamma\cup\overline{\Gamma})\times (\Gamma\cup\overline{\Gamma})$. Therefore, we have $G_{\mu}=G_{\nu}$ on the entire $(\mathbb{C}\setminus\mathbb{R})^{2}$ by analyticity.
The fact that $\mu$ and $\nu$ have the same Cauchy transform yields
the result.
\end{proof}
We now present a continuity theorem for the bi-free $R$-transform.
\begin{prop}
Let $\{\mu_{n}\}_{n=1}^{\infty}$ be a sequence of probability measures on
$\mathbb{R}^{2}$. Then $\mu_{n}$ converges weakly to a probability
measure on $\mathbb{R}^{2}$ if and only if
\begin{enumerate}[$\qquad(1)$]
\item there exists a Stolz angle $\Delta$ such that all $R_{\mu_{n}}$
are defined in the product domain $\Omega=(\Delta\cup\overline{\Delta})\times(\Delta\cup\overline{\Delta})$;
\item the pointwise limit $\lim_{n\rightarrow\infty}R_{\mu_{n}}(z,w)=R(z,w)$
exists for every $(z,w)$ in the domain $\Omega$; and
\item the limit $R_{\mu_{n}}(-iy,-iv)\rightarrow0$ holds uniformly in $n$
as $y,v\rightarrow0^{+}$.
\end{enumerate}

Moreover, if $\mu_{n}\Rightarrow\mu$, then we have $R=R_{\mu}$.

\end{prop}
\begin{proof}
Suppose $\mu_{n}\Rightarrow\mu$. Then we have the weak convergence
$\mu_{n}^{(j)}\Rightarrow\mu^{(j)}$ ($j=1,2$) for the marginal laws, because each projection $\pi_j$ is continuous.
By the continuity results for one-dimensional $R$-transform \cite{BerVoicu93}, this marginal weak convergence implies that there exists a Stolz
angle $\Delta$ such that all $R_{\mu_{n}^{(j)}}$ ($n\geq1$) and
$R_{\mu^{(j)}}$ are defined in $\Delta\cup\overline{\Delta}$, the
pointwise convergence $R_{\mu_{n}^{(j)}}\rightarrow R_{\mu^{(j)}}$
holds in $\Delta\cup\overline{\Delta}$ as $n\rightarrow\infty$,
and $\lambda R_{\mu_{n}^{(j)}}(\lambda)\rightarrow0$ uniformly in
$n$ as $\lambda\rightarrow0$ within the set $\Delta\cup\overline{\Delta}$.
The last uniform convergence result for $\lambda R_{\mu_{n}^{(j)}}(\lambda)$
amounts to $1/\lambda+R_{\mu_{n}^{(j)}}(\lambda)=(1/\lambda)(1+o(1))$
uniformly in $n$ as $\lambda\rightarrow0$, $\lambda\in\Delta\cup\overline{\Delta}$.
Thus, we conclude that $1/\lambda+R_{\mu_{n}^{(j)}}(\lambda)\rightarrow_{\leftslice}\infty$
uniformly in $n$ as $\lambda\rightarrow0$, $\lambda\in\Delta\cup\overline{\Delta}$. By shrinking the Stolz angle $\Delta$ if necessary but without changing the notation,  Proposition 2.2 (ii) and the definition (2.3) show that all $R_{\mu_n}$ are defined on the domain $\Omega=(\Delta\cup\overline{\Delta})\times(\Delta\cup\overline{\Delta})$, which is the statement (1). The statements (2) and (3) also follow from Proposition
2.2, with the limit function $R=R_{\mu}$.

Conversely, assume that (1) to (3) hold. The uniform limit condition
(3) implies that to each $\varepsilon>0$, there exists a small $\delta=\delta(\varepsilon)>0$
such that
\[
\left|R_{\mu_{n}}(-iy,-iv)\right|<\varepsilon,\qquad n\geq1,\quad0<y,v<\delta.
\]
By taking $v\rightarrow0$ and fixing $y$ in this inequality, Lemma
2.4 shows that $(-iy)R_{\mu_{n}^{(1)}}(-iy)\rightarrow0$ uniformly
in $n$ as $y\rightarrow0^{+}$. Therefore, again by the results in \cite{BerVoicu93}, the sequence $\{\mu_{n}^{(1)}\}_{n=1}^{\infty}$ is tight. In
the same way, we see that $\{\mu_{n}^{(2)}\}_{n=1}^{\infty}$ is also
a tight sequence. These two facts together imply the tightness of
$\{\mu_{n}\}_{n=1}^{\infty}$. By the first part of the proof and
Proposition 2.5, any weak limit $\mu$ of $\{\mu_{n}\}_{n=1}^{\infty}$
will be uniquely determined by the pointwise convergence condition
(2). Therefore, the full sequence $\mu_{n}$ must converge weakly
to $\mu$.
\end{proof}
Finally, we remark that the uniform condition (3) in the preceding result
is really the tightness of the sequence $\{\mu_{n}\}_{n=1}^{\infty}$
in disguise.

\section{limit theorems and infinite divisibility}

We now develop the theory of infinitely divisible $R$-transforms.
Consider an arbitrary sequence $\{\mu_{n}\}_{n=1}^{\infty}$ of probabilities
on $\mathbb{R}^{2}$, and let $k_{n}$ be any sequence of positive
integers tending to infinity. To motivate our discussion, assume
for the moment that each $\mu_{n}$ is compactly supported, then we
can view it as the common distribution for the finite sequence $(a_{n1},b_{n1}),(a_{n2},b_{n2}),\cdots,(a_{nk_{n}},b_{nk_{n}})$ of identically distributed bi-free two-faced pairs of commuting random variables. The theme of our investigation has to do with the following question:
\begin{problem}
What is the class of all possible distributional limits for the sum
\[
S_{n}=(a_{n1},b_{n1})+(a_{n2},b_{n2})+\cdots+(a_{nk_{n}},b_{nk_{n}})=(\textstyle \sum\nolimits_{j=1}^{k_n} a_{nj}, \textstyle \sum \nolimits_{j=1}^{k_n} b_{nj}),
\]
and what are the conditions for the law of $S_{n}$ to converge to
a specified limit distribution?
\end{problem}
Denote by $\nu_{n}$ the distribution of the sum $S_{n}$. Note that
Voiculescu's $R$-transform machinery shows that $R_{\nu_{n}}=k_{n}R_{\mu_{n}}$
in a bidisk centered at $(0,0)$. Several necessary conditions
for the convergence of $\{\nu_{n}\}_{n=1}^{\infty}$ are easy to derive.
First, if the sequence $\nu_{n}$ should converge weakly to a probability
law $\nu$ on $\mathbb{R}^{2}$, then the measures $\mu_{n}$ must
satisfy the following \emph{infinitesimality} condition: $\mu_{n}\Rightarrow\delta_{(0,0)}$.
Indeed, Proposition 2.6 shows that there is a universal domain of definition $\Omega$ for all $R_{\nu_n}$ (and hence for all $R_{\mu_n}$) such that $R_{\mu_{n}}=R_{\nu_{n}}/k_{n}=(R_{\nu}+o(1))\cdot o(1)$
as $n\rightarrow\infty$ in $\Omega$ and $R_{\mu_{n}}(-iy,-iv)=R_{\nu_{n}}(-iy,-iv)/k_{n}=o(1)$
uniformly in $n$ as $y,v\rightarrow 0^{+}$. So, we have $\mu_{n}\Rightarrow\delta_{(0,0)}$.
Of course, this also yields $\mu_{n}^{(j)}=\mu_{n}\circ\pi_{j}^{-1}\Rightarrow\delta_{0}$
($j=1,2$) for the marginal laws.
In fact, the converse of this is also true, that is, if the marginal
infinitesimality $\mu_{n}^{(j)}\Rightarrow\delta_{0}$ holds for $j=1,2$,
then one has $\mu_{n}\Rightarrow\delta_{(0,0)}$.

Secondly, to each $j$, we observe the weak convergence $\nu_{n}^{(j)}\Rightarrow\nu^{(j)}$
for the marginal laws, and by Lemma 2.4, we have
\begin{equation}
\nu_{n}^{(j)}=\mu_{n}^{(j)}\boxplus\mu_{n}^{(j)}\boxplus\cdots\boxplus\mu_{n}^{(j)}\qquad(k_{n}\;\text{times}),\label{eq:3.1}
\end{equation}
where $\boxplus$ is the usual free convolution for measures on $\mathbb{R}$.
By the Bercovici-Pata bijection \cite{BerPata99}, this means that each marginal
limit law $\nu^{(j)}$ must be $\boxplus$-infinitely divisible. On
the other hand, by applying the one-dimensional $R$-transform to
the weak convergence $\nu_{n}^{(j)}\Rightarrow\nu^{(j)}$, we get
$\mu_{n}^{(j)}\Rightarrow\delta_{0}$.

Thus, assuming the weak convergence of the marginals $\mu_{n}^{(j)}$ under
free convolution, the key to the solution of Problem 3.1 is Theorem
3.2 below, in which the measures $\mu_{n}$ are no longer assumed
to be compactly supported.

In order to prove Theorem 3.2, we now review the limit theorems proved
in \cite{BerPata99} for free convolution $\boxplus$ on $\mathbb{R}$.
The one-dimensional $R$-transform of an $\boxplus$-infinitely divisible law
$\nu$ on $\mathbb{R}$ admits a \emph{free L\'{e}vy-Khintchine representation}:
\[
R_{\nu}(z)=\gamma+\int_{\mathbb{R}}\frac{z+x}{1-zx}\, d\sigma(x),\qquad z\notin \mathbb{R},
\]
where $\gamma\in\mathbb{R}$ and $\sigma$ is a finite Borel measure
on $\mathbb{R}$ (called \emph{L\'{e}vy parameters}). The pair $(\gamma,\sigma)$
is unique. Conversely, given L\'{e}vy parameters $\gamma$ and $\sigma$,
this integral formula determines a unique $\boxplus$-infinitely divisible
law $\nu$ on $\mathbb{R}$. We shall write $\nu=\nu_{\boxplus}^{\gamma,\sigma}$
to indicate this correspondence. In order for the free convolutions
(\ref{eq:3.1}) to converge weakly to
$\nu_{\boxplus}^{\gamma,\sigma}$ on $\mathbb{R}$, it is necessary
and sufficient that the limit
\[
\lim_{n\rightarrow\infty}\int_{\mathbb{R}}\frac{k_{n}x}{1+x^{2}}\, d\mu_{n}^{(j)}(x)=\gamma
\]
and the one-dimensional weak convergence
\begin{equation}
\frac{k_{n}x^{2}}{1+x^{2}}\, d\mu_{n}^{(j)}(x)\Rightarrow\sigma\label{eq:3.2}
\end{equation}
hold simultaneously.
\begin{thm}
Let $\{\mu_{n}\}_{n=1}^{\infty}$ be a sequence of probability measures
on $\mathbb{R}^{2}$, and let $k_{n}$ be a sequence of positive integers
such that $\lim_{n\rightarrow\infty}k_{n}=\infty$. For $j=1,2$,
assume that $\nu_{jn}=[\mu_{n}^{(j)}]^{\boxplus k_{n}}$ converges
weakly to $\nu_{\boxplus}^{\gamma_{j},\sigma_{j}}$, the $\boxplus$-infinitely
divisible law determined by L\'{e}vy parameters $(\gamma_{j},\sigma_{j})$.
Then the following statements are equivalent:
\begin{enumerate}[$\qquad(1)$]
\item The pointwise limit $R(z,w)=\lim_{n\rightarrow\infty}k_{n}R_{\mu_{n}}(z,w)$
exists for $(z,w)$ in a domain $\Omega$ where all $R_{\mu_{n}}$
are defined.
\item The pointwise limit
\[
D(z,w)=\lim_{n\rightarrow\infty}k_{n}\int_{\mathbb{R}^{2}}\frac{zwst}{(1-zs)(1-wt)}\, d\mu_{n}(s,t)
\]
 exists for any $(z,w)\in(\mathbb{C}\setminus\mathbb{R})^{2}$.
\item The finite signed measures
\[
d\rho_{n}(s,t)=k_{n}\,\frac{st}{\sqrt{1+s^{2}}\sqrt{1+t^{2}}}\, d\mu_{n}(s,t)
\]
converge weakly to a finite signed measure $\rho$ on $\mathbb{R}^{2}$.
\end{enumerate}

Moreover, if \emph{(1)}, \emph{(2)}, and \emph{(3)} hold, then the
limit function $D$ has a unique integral representation
\[
D(z,w)=\int_{\mathbb{R}^{2}}\frac{zw\sqrt{1+s^{2}}\sqrt{1+t^{2}}}{(1-zs)(1-wt)}\, d\rho(s,t),
\]
and we have
\[
R(z,w)=zR_{\nu_{\boxplus}^{\gamma_{1},\sigma_{1}}}(z)+wR_{\nu_{\boxplus}^{\gamma_{2},\sigma_{2}}}(w)+D(z,w).
\]
In particular, the limit $R(z,w)$ extends analytically to $(\mathbb{C}\setminus\mathbb{R})^{2}$.

\end{thm}
\begin{proof}
We have seen that the infinitesimality of $\{\mu_{n}\}_{n=1}^{\infty}$
follows from the weak convergence of $\{\nu_{jn}\}_{n=1}^{\infty}$
($j=1,2$). Let $\Omega$ be the universal domain of definition for
all $R_{\mu_{n}}$, whose existence is guaranteed by Proposition 2.6.
By the definition \eqref{eq:2.3} of $R$-transform and the assumption $\nu_{jn}\Rightarrow\nu_{\boxplus}^{\gamma_{j},\sigma_{j}}$,
the pointwise limit $R(z,w)$ exists at $(z,w)\in\Omega$ if and only
if the limit $\lim_{n\rightarrow\infty}k_{n}[1-1/h_{\mu_{n}}(z,w)]$
does. Moreover, since $\lim_{n\rightarrow\infty}h_{\mu_{n}}(z,w)=1$
by Proposition 2.2, we conclude that the limit $R(z,w)$ exists if
and only if the limit
\begin{equation}
h(z,w)=\lim_{n\rightarrow\infty}k_{n}[h_{\mu_{n}}(z,w)-1]\label{eq:3.3}
\end{equation}
exists in $\Omega$, and in this case we will have
\[
R(z,w)=zR_{\nu_{\boxplus}^{\gamma_{1},\sigma_{1}}}(z)+wR_{\nu_{\boxplus}^{\gamma_{2},\sigma_{2}}}(w)+h(z,w).
\]

We first show that the limit $h(z,w)$ in (\ref{eq:3.3}) is
in fact equal to the limit $D(z,w)$. For simplicity, we set $G_{jn}^{-1}=G_{\mu_{n}^{(j)}}^{-1}$
and $R_{jn}=R_{\mu_{n}^{(j)}}$. According to the identities
\[
\frac{1}{1-zs+zR_{1n}(z)}=\frac{1}{1-zs}\left[1-\frac{R_{1n}(z)}{G_{1n}^{-1}(z)-s}\right]
\]
and
\[
\frac{1}{1-wt+wR_{2n}(w)}=\frac{1}{1-wt}\left[1-\frac{R_{2n}(w)}{G_{2n}^{-1}(w)-t}\right],
\]
we can write
\begin{eqnarray*}
h_{\mu_{n}}(z,w)-1 & = & -zR_{1n}(z)\int_{\mathbb{R}^{2}}\frac{1}{(1-zs)(1-wt)(zG_{1n}^{-1}(z)-zs)}\, d\mu_{n}(s,t)\\
 &  & -wR_{2n}(w)\int_{\mathbb{R}^{2}}\frac{1}{(1-zs)(1-wt)(wG_{2n}^{-1}(w)-wt)}\, d\mu_{n}(s,t)\\
 &  & +\int_{\mathbb{R}^{2}}\frac{zwR_{1n}(z)R_{2n}(w)}{(1-zs)(1-wt)(zG_{1n}^{-1}(z)-zs)(wG_{2n}^{-1}(w)-wt)}\, d\mu_{n}(s,t)\\
 &  & +\int_{\mathbb{R}^{2}}\frac{zs+wt-zwst}{(1-zs)(1-wt)}\, d\mu_{n}(s,t).
\end{eqnarray*}
As $n\rightarrow\infty$, we observe that $k_{n}R_{jn}=R_{\nu_{jn}}=R_{\nu_{\boxplus}^{\gamma_{j},\sigma_{j}}}+o(1)$,
$R_{jn}=o(1)$, and all the integrals above are of order $(1+o(1))$
except for the last one. Meanwhile, we have
\[
\int_{\mathbb{R}}\frac{s}{1-zs}\, d\mu_{n}^{(1)}(s)=R_{1n}(z)\cdot(1+o(1))
\]
and
\[
\int_{\mathbb{R}}\frac{t}{1-wt}\, d\mu_{n}^{(2)}(t)=R_{2n}(w)\cdot(1+o(1))
\]
as $n\rightarrow\infty$ (see \cite{BerPata99}), as well as the following
decomposition
\begin{eqnarray*}
\int_{\mathbb{R}^{2}}\frac{zs+wt-zwst}{(1-zs)(1-wt)}\, d\mu_{n}(s,t) & = & z\int_{\mathbb{R}}\frac{s}{1-zs}\, d\mu_{n}^{(1)}(s)+w\int_{\mathbb{R}}\frac{t}{1-wt}\, d\mu_{n}^{(2)}(t)\\
 &  & +\int_{\mathbb{R}^{2}}\frac{zwst}{(1-zs)(1-wt)}\, d\mu_{n}(s,t).
\end{eqnarray*}
We conclude from these findings that
\[
k_{n}[h_{\mu_{n}}(z,w)-1]=k_{n}\int_{\mathbb{R}^{2}}\frac{zwst}{(1-zs)(1-wt)}\, d\mu_{n}(s,t)+o(1)
\]
as $n\rightarrow\infty$. It follows that the limit $h(z,w)$ exists
if and only if the limit $D(z,w)$ does and $h(z,w)=D(z,w)$, as desired.
The equivalence between (1) and (2) is proved, at least for $(z,w)\in\Omega$.
At the end of this proof, we shall see that if the limit $D$ exists
in $\Omega$ then it also exists in the whole space $(\mathbb{C}\setminus\mathbb{R})^{2}$.

Next, we show that the existence of the limit $D$ in $\Omega$ implies
(3). Note that the total variation $\left|\rho_{n}\right|$ of the measure $\rho_{n}$
is given by
\[
d|\rho_{n}|(s,t)=k_{n}\,\frac{|st|}{\sqrt{1+s^{2}}\sqrt{1+t^{2}}}\, d\mu_{n}(s,t).
\] To each $n\geq1$, we introduce the positive measures
\[
d\sigma_{n}^{(1)}(s)=\frac{k_{n}s^{2}}{1+s^{2}}\, d\mu_{n}^{(1)}(s)\quad\text{and}\quad d\sigma_{n}^{(2)}(t)=\frac{k_{n}t^{2}}{1+t^{2}}\, d\mu_{n}^{(2)}(t),
\]
on $\mathbb{R}$ and note that the weak convergence condition (\ref{eq:3.2})
implies that both families $\{\sigma_{n}^{(j)}\}_{n=1}^{\infty}$
($j=1,2$) are tight and uniformly bounded in total variation norm.
Applying the Cauchy-Schwarz inequality to the measure $k_{n}d\mu_{n}$,
we obtain
\begin{eqnarray*}
[|\rho_{n}|(\mathbb{R}^{2})]^{2} & \leq & \int_{\mathbb{R}^{2}}\frac{s^{2}}{1+s^{2}}\, k_{n}d\mu_{n}(s,t)\cdot\int_{\mathbb{R}^{2}}\frac{t^{2}}{1+t^{2}}\, k_{n}d\mu_{n}(s,t)\\
 & = & \sigma_{n}^{(1)}(\mathbb{R})\sigma_{n}^{(2)}(\mathbb{R})\leq2\sigma_{1}(\mathbb{R})\sigma_{2}(\mathbb{R})<\infty.
\end{eqnarray*}
Hence, the family $\{\rho_{n}\}_{n=1}^{\infty}$ is bounded in total
variation norm.

On the other hand, for any closed square $K_{m}=\{(s,t):|s|\leq m,|t|\leq m\}$,
we have
\begin{equation}
[|\rho_{n}|(\mathbb{R}^{2}\setminus K_{m})]^{2}\leq\int_{\mathbb{R}^{2}\setminus K_{m}}\frac{k_{n}s^{2}}{1+s^{2}}\, d\mu_{n}(s,t)\cdot\int_{\mathbb{R}^{2}\setminus K_{m}}\frac{k_{n}t^{2}}{1+t^{2}}\, d\mu_{n}(s,t)\label{eq:3.4}
\end{equation}
by the Cauchy-Schwarz inequality again. Next, observe that
\begin{eqnarray*}
\int_{\{(s,t):|s|\leq m,|t|>m\}}\frac{k_{n}s^{2}}{1+s^{2}}\, d\mu_{n}(s,t) & = & \int_{\{(s,t):|s|\leq m,|t|>m\}}\frac{s^{2}(1+t^{2})}{(1+s^{2})t^{2}}\frac{k_{n}t^{2}}{1+t^{2}}\, d\mu_{n}(s,t)\\
 & \leq & (1+1/m^{2})\int_{\{(s,t):|t|>m\}}\frac{k_{n}t^{2}}{1+t^{2}}\, d\mu_{n}(s,t)\\
 & = & (1+1/m^{2})\sigma_{n}^{(2)}(\mathbb{R}\setminus[-m,m]).
\end{eqnarray*}
It follows that
\begin{eqnarray*}
\int_{\mathbb{R}^{2}\setminus K_{m}}\frac{k_{n}s^{2}}{1+s^{2}}\, d\mu_{n}(s,t) & = & \int_{\{(s,t):|s|>m\}}\frac{k_{n}s^{2}}{1+s^{2}}\, d\mu_{n}(s,t)\\
 &  & +\int_{\{(s,t):|s|\leq m,|t|>m\}}\frac{k_{n}s^{2}}{1+s^{2}}\, d\mu_{n}(s,t)\\
 & \leq & \sup_{n\geq1}\sigma_{n}^{(1)}(\mathbb{R}\setminus[-m,m])\\
 &  & +(1+1/m^{2})\sup_{n\geq1}\sigma_{n}^{(2)}(\mathbb{R}\setminus[-m,m])\\
 & \longrightarrow & 0
\end{eqnarray*}
uniformly in $n$ as $m\rightarrow\infty$, by tightness. Likewise, we also have
\[
\lim_{m\rightarrow\infty}\sup_{n\geq1}\int_{\mathbb{R}^{2}\setminus K_{m}}\frac{k_{n}t^{2}}{1+t^{2}}\, d\mu_{n}(s,t)=0.
\]
By virtue of (\ref{eq:3.4}), these uniform limits imply the tightness
of the signed measures $\{\rho_{n}\}_{n=1}^{\infty}$. Consequently,
the measures $\{\rho_{n}\}_{n=1}^{\infty}$ have weak limit points.

Now, suppose that $\rho$ and $\rho^{\prime}$ are both weak limits
for the sequence $\{\rho_{n}\}_{n=1}^{\infty}$ . We will argue that
$\rho=\rho^{\prime}$ and hence the entire sequence $\{\rho_{n}\}_{n=1}^{\infty}$
must converge weakly to $\rho$. Toward this end we examine the limit
$D$:
\[
D(z,w)=\lim_{n\rightarrow\infty}\int_{\mathbb{R}^{2}}\frac{\sqrt{1+s^{2}}\sqrt{1+t^{2}}}{(1/z-s)(1/w-t)}\, d\rho_{n}(s,t).
\]
We deduce from this identity that the signed measures $\sqrt{1+s^{2}}\sqrt{1+t^{2}}\, d\rho(s,t)$
and $\sqrt{1+s^{2}}\sqrt{1+t^{2}}\, d\rho^{\prime}(s,t)$ have the
same Cauchy transforms, first on the open set $\{(z,w):(1/z,1/w)\in\Omega\}$
and thus to everywhere by the analyticity of Cauchy transform. This
yields $\rho=\rho^{\prime}$, finishing the proof of (2) implying
(3) under the assumption that $D$ exists in $\Omega$. This
argument also shows the uniqueness of the integral representation
for the function $D$.

Finally, if (3) holds then the limit $D$ exists not only in $\Omega$
but also on the entire space $(\mathbb{C}\setminus\mathbb{R})^{2}$,
because
\[
\frac{zw\sqrt{1+s^{2}}\sqrt{1+t^{2}}}{(1-zs)(1-wt)}
\]
is always a bounded and continuous function in $(s,t)$, as long as
$z,w\notin\mathbb{R}$. This theorem is completely proved.
\end{proof}
Suppose that the sequences $\{\mu_{n}\}_{n=1}^{\infty}$ and $\{k_{n}\}_{n=1}^{\infty}$
have an additional property that
\[
k_{n}R_{\mu_{n}}=R_{\nu_{n}},\qquad n\geq1,
\]
for some probability law $\nu_{n}$ on $\mathbb{R}^{2}$. For example,
this occurs for any $k_{n}$ when $\mu_{n}$ is compactly supported,
and the resulting measure $\nu_{n}$ would be the $k_{n}$-th bi-free
convolution power of $\mu_{n}$. In this case it makes sense to investigate
the convergence of the laws $\{\nu_{n}\}_{n=1}^{\infty}$, and we
have the following result which provides an answer to the second part
of Problem 3.1. Recall that the signed measures $\{\rho_{n}\}_{n=1}^{\infty}$
are defined as in Theorem 3.2 (3).
\begin{cor}
\emph{(Convergence Criteria)} The sequence $\nu_{n}$ converges weakly
to a probability law on $\mathbb{R}^{2}$ if and only if the marginal
free convolutions $[\mu_{n}^{(j)}]^{\boxplus k_{n}}$ and the signed
measures $\rho_{n}$ converge weakly on $\mathbb{R}$ and $\mathbb{R}^{2}$,
respectively. Furthermore, if $\nu_{n}\Rightarrow\nu$, $[\mu_{n}^{(j)}]^{\boxplus k_{n}}\Rightarrow\nu_{\boxplus}^{\gamma_{j},\sigma_{j}}$
($j=1,2$), and $\rho_{n}\Rightarrow\rho$, then we have the marginal law
$\nu^{(j)}=\nu_{\boxplus}^{\gamma_{j},\sigma_{j}}$ for $j=1,2$,
and
\begin{equation}
G_{\nu}(z,w)\left[1-G_{_{\sqrt{1+s^{2}}\sqrt{1+t^{2}}\, d\rho}}(1/G_{\nu^{(1)}}(z),1/G_{\nu^{(2)}}(w))\right]=G_{\nu^{(1)}}(z)G_{\nu^{(2)}}(w)\label{eq:3.5}
\end{equation}
for $(z,w)\in(\mathbb{C}\setminus\mathbb{R})^{2}$.\end{cor}
\begin{proof}
Assume $\nu_{n}\Rightarrow\nu$ for some probability law $\nu$ on
$\mathbb{R}^{2}$. Since $\nu_{n}^{(j)}=[\mu_{n}^{(j)}]^{\boxplus k_{n}}$
by Lemma 2.4, we have $[\mu_{n}^{(j)}]^{\boxplus k_{n}}\Rightarrow\nu^{(j)}$.
It follows that the limit law $\nu^{(j)}$ is $\boxplus$-infinitely
divisible and $\mu_{n}\Rightarrow\delta_{(0,0)}$. The weak convergence
of the measures $\rho_{n}$ is then a consequence of the pointwise
convergence $R_{\nu_{n}}\rightarrow R_{\nu}$ by Theorem 3.2.

Conversely, assume the weak convergence of $\{[\mu_{n}^{(j)}]^{\boxplus k_{n}}\}_{n=1}^{\infty}$
and that of $\{\rho_{n}\}_{n=1}^{\infty}$. The marginal weak convergence
implies that $\{\nu_{n}\}_{n=1}^{\infty}$ is tight, and hence the
condition (3) of Proposition 2.6 holds for the sequence $\{R_{\nu_{n}}\}_{n=1}^{\infty}$.
To conclude, we only need to verify the pointwise convergence of $\{R_{\nu_{n}}\}_{n=1}^{\infty}$.
This property, however, is equivalent to the weak convergence of $\{\rho_{n}\}_{n=1}^{\infty}$
by Theorem 3.2. Therefore, $\{\nu_{n}\}_{n=1}^{\infty}$ is a weakly
convergent sequence.

Finally, the functional equation \eqref{eq:3.5} follows from the
fact that the $R$-transform of the limit $\nu$ has the integral
representation
\[
R_{\nu}(z,w)=zR_{\nu_{\boxplus}^{\gamma_{1},\sigma_{1}}}(z)+wR_{\nu_{\boxplus}^{\gamma_{2},\sigma_{2}}}(w)+G_{_{\sqrt{1+s^{2}}\sqrt{1+t^{2}}\, d\rho}}(1/z,1/w)
\]
for $z,w\notin\mathbb{R}$.
\end{proof}
Next, we present some examples in which the limit laws are constructed
via the central limit process or the Poisson type limit theorems.
We are mainly interested in probability measures that are \emph{full}
in the sense that they are not supported on a one-dimensional line
in $\mathbb{R}^{2}$.
\begin{example}
(Bi-free Gaussians) These are the limit laws $\nu$ with $\sigma_{1}=a\delta_{0}$,
$\sigma_{2}=b\delta_{0}$, and $\rho=c\delta_{(0,0)}$, where $a,b>0$
and $|c|\leq\sqrt{ab}$. The vector $(\gamma_{1},\gamma_{2})$ represents
the mean of the law $\nu$, and
\[
\left(\begin{array}{cc}
a & c\\
c & b
\end{array}\right)
\]
is the covariance matrix of $\nu$. The $R$-transform of $\nu$ is
given by
\[
R_{\nu}(z,w)=\gamma_{1}z+\gamma_{2}w+az^{2}+bw^{2}+czw.
\]
The marginal law $\nu^{(1)}$ is the semicircular law with mean $\gamma_{1}$
and variance $a$, and the law $\nu^{(2)}$ is the same with mean
$\gamma_{2}$ and variance $b$. The equation (\ref{eq:3.5}) shows
that the law $\nu$ is compactly supported and absolutely continuous
with respect to the Lebesgue measure $ds\,dt$ on $\mathbb{R}^{2}$,
whenever $|c|<\sqrt{ab}$. In the standardized case of $\gamma_{1}=\gamma_{2}=0$,
$a=b=1$, and $|c|<1$, the inversion formula described in Section 2.1 gives the
following density formula:
\[
d\nu=\frac{1-c^{2}}{2\pi^{2}}\frac{\sqrt{4-s^{2}}\sqrt{4-t^{2}}}{2(1-c^{2})^{2}-c(1+c^{2})st+2c^{2}(s^{2}+t^{2})}\, ds\, dt,
\]
where $s,t \in [-2,2]$. In particular, if the marginals are uncorrelated
(i.e., $c=0$), then we have $\nu=\mathcal{S}\otimes\mathcal{S}$,
the product of the standard semicircle law $\mathcal{S}$. The degenerate
case $|c|=1$ corresponds to a non-full probability measure concentrated
entirely on a straight line in the plane. Thus, such a degenerate
law is purely singular to the Lebesgue measure on $\mathbb{R}^{2}$.
The existence of bi-free Gaussian laws is provided by central limit
theorems. Indeed, given $a=b=1$ and a correlation coefficient $c\in[-1,1]$,
let $Z_{1}$ and $Z_{2}$ be two classically independent real-valued
random variables drawn from the same law $(1/2)\delta_{-1}+(1/2)\delta_{1}$,
and let $\mu_{n}$ be the distribution of the random vector
\[
(X_{n},Y_{n})=(\sqrt{(1+c)/2n}Z_{1}-\sqrt{(1-c)/2n}Z_{2},\sqrt{(1+c)/2n}Z_{1}+\sqrt{(1-c)/2n}Z_{2}).
\]
Since the normal domain of attraction of the standard Gaussian law
coincides with that of $\mathcal{S}$, the marginal weak convergence
$[\mu_{n}^{(j)}]^{\boxplus n}\Rightarrow\mathcal{S}$ holds for $j=1,2$.
On the other hand, the pointwise limit
\[
D(z,w)=\lim_{n\rightarrow\infty}n\, E\left[\frac{zwX_{n}Y_{n}}{(1-zX_{n})(1-wY_{n})}\right]=czw=G_{c\delta_{(0,0)}}(1/z,1/w)
\]
implies the weak convergence $\rho_{n}\Rightarrow c\delta_{(0,0)}$.
By Corollary 3.3, the measures $\nu_{n}$ converge to the bi-free
Gaussian with zero mean and covariance
\[
\left(\begin{array}{cc}
1 & c\\
c & 1
\end{array}\right).
\]
The general case of $a$ and $b$ follows from a simple rescaling argument.
\begin{example}
(Bi-free compound Poisson laws) Let $\mu\neq\delta_{(0,0)}$ be
a probability measure on $\mathbb{R}^{2}$ and let $\lambda>0$ be
a given parameter. Consider the sequence
\[
\mu_{n}=(1-\lambda/n)\delta_{(0,0)}+(\lambda/n)\mu,\qquad n\geq1,
\]
and the corresponding marginals $\mu_{n}^{(j)}=(1-\lambda/n)\delta_{0}+(\lambda/n)\mu^{(j)}$.
It is easy to see that the marginal free convolutions $[\mu_{n}^{(j)}]^{\boxplus n}$
converge weakly to the usual free compound Poisson law with the
L\'{e}vy parameters
\[
\gamma_{j}=\lambda\int_{\mathbb{R}}\frac{x}{1+x^{2}}\, d\mu^{(j)}(x),\quad d\sigma_{j}(x)=\frac{\lambda x^{2}}{1+x^{2}}\, d\mu^{(j)}(x),
\]
and that the signed measures $\rho_{n}$ converge weakly on $\mathbb{R}^{2}$
to
\[
\rho=\frac{\lambda st}{\sqrt{1+s^{2}}\sqrt{1+t^{2}}}\, d\mu(s,t).
\]
So, by Theorem 3.2, the limit $R=\lim_{n\rightarrow\infty}nR_{\mu_{n}}$
exists and has the integral representation $R(z,w)=\lambda\left[(1/zw)G_{\mu}(1/z,1/w)-1\right]$ or, equivalently,
\[
R(z,w)=-\lambda+\lambda\int_{\mathbb{R}^{2}}\frac{1}{(1-zs)(1-wt)}\, d\mu(s,t),\qquad z,w\notin\mathbb{R}.
\]
The last integral is in fact the bi-free $R$-transform of a unique probability distribution $\nu_{\lambda,\mu}$, called the \emph{bi-free compound
Poisson law} with rate $\lambda$ and jump distribution $\mu$, on
$\mathbb{R}^{2}$. (The \emph{bi-free Poisson law} with rate $\lambda$
is defined to be the law $\nu_{\lambda,\delta_{(1,1)}}$.) We now
verify the existence of $\nu_{\lambda,\mu}$ by the method of truncation.
For each positive integer $m$, let $f_{m}$ be a compactly supported,
continuous function such that $0\leq f_{m}\leq1$, $f_{m}(s,t)=1$
for $(s,t)\in K_{m}=\{(s,t):|s|\leq m,|t|\leq m\}$, and $f_{m}$
vanishes on the complement $\mathbb{R}^2\setminus K_{m+1}$. For sufficiently large $m$, we introduce the following truncation
\[
d\mu^{m}=c_{m}\,f_{m}\, d\mu,
\] where the normalization constant $c_{m}$ is chosen so that $\mu^{m}(\mathbb{R}^{2})=1$. We observe from an application of the dominated convergence
theorem that $c_{m}\rightarrow 1$ and the weak convergence $\mu^{m}\Rightarrow\mu$ on $\mathbb{R}^{2}$ as $m\rightarrow\infty$. After applying the above limiting process to each truncation
$\mu^{m}$, Corollary 3.3 provides a limit law $\nu_{\lambda,\mu^{m}}$
whose $R$-transform is given by
\[
R_{\nu_{\lambda,\mu^{m}}}(z,w)=\lambda\left[(1/zw)G_{\mu^{m}}(1/z,1/w)-1\right],\qquad z,w\notin\mathbb{R}.
\]
By Proposition 2.2, the weak convergence $\mu^{m}\Rightarrow\mu$
implies that $\{R_{\nu_{\lambda,\mu^{m}}}\}_{m=1}^{\infty}$ converges
pointwisely to the integral $R$ and $R_{\nu_{\lambda,\mu^{m}}}(-iy,-iv)\rightarrow0$
uniformly in $m$ as $y,v\rightarrow0^{+}$. Therefore, by Proposition
2.6, these conditions show further that the laws $\nu_{\lambda,\mu^{m}}$
converge weakly to a unique limit distribution $\nu_{\lambda,\mu}$
and $R=R_{\nu_{\lambda,\mu}}$, as desired.
\end{example}
\end{example}
We continue our investigation on the characterization of the limit
laws. Let $R=\lim_{n\rightarrow\infty}k_{n}R_{\mu_{n}}$ be the pointwise
limit in Theorem 3.2. Then for any integer $m\geq2$, the function
$R/m$ is the limit of $[k_{n}/m]R_{\mu_{n}}$ where $[x]$ indicates
the integral part of $x\in\mathbb{R}$. In other words, the limit
of $k_{n}R_{\mu_{n}}$ can always be decomposed into a sum of $m$
identical functions of the same kind. This nature of the limit $R$
can also be seen from its integral representation, say,
\[
R(z,w)/m=zR_{\nu_{\boxplus}^{\gamma_{1}/m,\sigma_{1}/m}}(z)+wR_{\nu_{\boxplus}^{\gamma_{2}/m,\sigma_{2}/m}}(w)+G_{_{\sqrt{1+s^{2}}\sqrt{1+t^{2}}\, d\rho/m}}(1/z,1/w),
\]
where the quintuple $(\gamma_{1},\gamma_{2},\sigma_{1},\sigma_{2},\rho)$
of numbers and measures are provided by Theorem 3.2. In the next result
we show that every limiting integral form of this kind is indeed a
bi-free $R$-transform.
\begin{thm}
Let $\{\mu_{n}\}_{n=1}^{\infty}$ and $\{k_{n}\}_{n=1}^{\infty}$
be two sequences of probability measures and positive integers satisfying
the hypotheses of \emph{Theorem 3.2}, and assume that the limit function
$ $$R=\lim_{n\rightarrow\infty}k_{n}R_{\mu_{n}}$ exists in a domain $\Omega$.
Then there exists a unique probability law $\nu$ on $\mathbb{R}^{2}$
such that $R=R_{\nu}$ on $\Omega$.\end{thm}
\begin{proof}
The uniqueness statement follows from Proposition 2.5; we shall prove
the existence of $\nu$. The idea of the proof is to consider a two-dimensional
``lifting'' of the integral representation of the limit $R$. First,
from the proof of Theorem 3.2 we see that both
\[
\left\{ \sigma_{1n}=\frac{k_{n}s^{2}}{1+s^{2}}\, d\mu_{n}(s,t):n\geq1\right\} \quad\text{and}\quad\left\{ \sigma_{2n}=\frac{k_{n}t^{2}}{1+t^{2}}\, d\mu_{n}(s,t):n\geq1\right\}
\]
are tight families of Borel measures on $\mathbb{R}^{2}$. By dropping
to subsequences but without changing the notations, we may and do
assume that they are both weakly convergent sequences of measures
on $\mathbb{R}^{2}$. We write
\[
\gamma_{1n}=\int_{\mathbb{R}^{2}}\frac{k_{n}s}{1+s^{2}}\, d\mu_{n}(s,t)\quad\text{and}\quad\gamma_{2n}=\int_{\mathbb{R}^{2}}\frac{k_{n}t}{1+t^{2}}\, d\mu_{n}(s,t)
\]
so that $\lim_{n\rightarrow\infty}\gamma_{jn}=\gamma_{j}$ for $j=1,2$.
Note that we have
\[
R_{\nu_{\boxplus}^{\gamma_{1},\sigma_{1}}}(z)=\lim_{n\rightarrow\infty}\left[\gamma_{1n}+\int_{\mathbb{R}^{2}}\frac{z+s}{1-zs}\, d\sigma_{1n}(s,t)\right]
\]
and a similar formula for the function $R_{\nu_{\boxplus}^{\gamma_{2},\sigma_{2}}}$.
Together with the weak convergence $\rho_{n}\Rightarrow\rho$, a straightforward
calculation leads to the following asymptotic Poissonization for the
limit $R$:
\begin{eqnarray*}
R(z,w) & = & zR_{\nu_{\boxplus}^{\gamma_{1},\sigma_{1}}}(z)+wR_{\nu_{\boxplus}^{\gamma_{2},\sigma_{2}}}(w)+\lim_{n\rightarrow\infty}k_{n}\int_{\mathbb{R}^{2}}\frac{zwst}{(1-zs)(1-wt)}\, d\mu_{n}(s,t)\\
 & = & \lim_{n\rightarrow\infty}k_{n}\left[(1/zw)G_{\mu_{n}}(1/z,1/w)-1\right]\\
 & = & \lim_{n\rightarrow\infty}R_{\nu_{k_{n},\mu_{n}}}(z,w)
\end{eqnarray*}
for $z,w\notin\mathbb{R}$.

The rest of the proof will be devoted to showing the tightness of
the compound Poisson laws $\{\nu_{k_{n},\mu_{n}}\}_{n=1}^{\infty}$;
that is, $R_{\nu_{k_{n},\mu_{n}}}(-iy,-iv)\rightarrow0$ uniformly
in $n$ as $y,v\rightarrow0^{+}$. Indeed, if this were true then
the limit $\nu$ of $\{\nu_{k_{n},\mu_{n}}\}_{n=1}^{\infty}$ would satisfy $R=R_{\nu}$ by Proposition 2.6, bringing us to the end of this
proof. To this purpose, consider first the estimates
\[
\left|\frac{(-iy)^{2}-iys}{1+iys}\right|\leq1,\qquad s\in\mathbb{R},
\]
and
\[
\left|\frac{-iy+s}{1+iys}\right|\leq\sqrt{2}\frac{y+\left|s\right|}{1+y\left|s\right|}\leq\sqrt{2}(1+T),\qquad\left|s\right|\leq T,
\]
for arbitrary $0<y\leq1$ and $T>0$. These observations imply that
\[
\sup_{n\geq1}\left|\int_{\mathbb{R}^{2}}\frac{(-iy)^{2}-iys}{1+iys}\, d\sigma_{1n}(s,t)\right|\leq\sqrt{2}(1+T)y+\sup_{n\geq1}\sigma_{1n}(\{(s,t):\left|s\right|>T\})
\]
for these $y$ and $T$. Since the tail-sums of the tight sequence
$\{\sigma_{1n}\}_{n=1}^{\infty}$ can be made uniformly small as we
wish and since the sequence $\{\gamma_{1n}\}_{n=1}^{\infty}$ is bounded,
we conclude that
\[
\lim_{y\rightarrow0^{+}}\sup_{n\geq1}\left|i\gamma_{1n}y+\int_{\mathbb{R}^{2}}\frac{(-iy)^{2}-iys}{1+iys}\, d\sigma_{1n}(s,t)\right|=0.
\]
Using the same method, it can be shown that
\[
\lim_{v\rightarrow0^{+}}\sup_{n\geq1}\left|i\gamma_{2n}v+\int_{\mathbb{R}^{2}}\frac{(-iv)^{2}-ivt}{1+ivt}\, d\sigma_{2n}(s,t)\right|=0
\]
and
\[
\lim_{y,v\rightarrow0^{+}}\sup_{n\geq1}\left|\int_{\mathbb{R}^{2}}\frac{(iy)(iv)\sqrt{1+s^{2}}\sqrt{1+t^{2}}}{(1+iys)(1+ivt)}\, d\rho_{n}(s,t)\right|=0.
\]
These uniform limits imply that $R_{\nu_{k_{n},\mu_{n}}}(-iy,-iv)$ tends to zero
uniformly in $n$ as $y,v\rightarrow0^{+}$, just as we expected. \end{proof}

Our results motivate the following definition.
\begin{defn}
A bi-free partial $R$-transform $R_{\nu}$ on a domain $\Omega$ is said to
be\emph{ $ $infinitely divisible} if for each positive integer $m\geq 2$, there exists
a probability law $\mu_{m}$ on $\mathbb{R}^{2}$ such that $R_{\nu}=mR_{\mu_{m}}$
in $\Omega$. In this case, the law $\nu$ is said to be \emph{bi-freely}
\emph{infinitely divisible}.
\end{defn}
\begin{example} (Product of $\boxplus$-infinitely divisible laws) All point masses are clearly bi-freely infinitely divisible. A less trivial way of constructing a bi-freely infinitely divisible law is to form the product measure of two $\boxplus$-infinitely divisible laws. Thus, given two sets of L\'{e}vy parameters $(\gamma_1,\sigma_1)$ and  $(\gamma_2,\sigma_2)$, the definition \eqref{eq:2.3} and the one-dimensional L\'{e}vy-Khintchine formula imply that the product measure \[\nu=\nu_{\boxplus}^{\gamma_1,\sigma_1}\otimes \nu_{\boxplus}^{\gamma_2,\sigma_2}\] on $\mathbb{R}^2$ has the bi-free $R$-transform \[R_{\nu}(z,w)=z\left[\gamma_{1}+\int_{\mathbb{R}}\frac{z+s}{1-zs}\, d\sigma_{1}(s)\right]+w\left[\gamma_{2}+\int_{\mathbb{R}}\frac{w+t}{1-wt}\, d\sigma_{2}(t)\right]\] for $(z,w)\in (\mathbb{C}\setminus \mathbb{R})^2$. It is obvious that $R_{\nu}/m$ is the bi-free $R$-transform of the corresponding product measure \[\nu_{\boxplus}^{\gamma_1/m,\sigma_1/m}\otimes \nu_{\boxplus}^{\gamma_2/m,\sigma_2/m}\] for all $m\geq 2$.
\end{example}

Theorem 3.6 shows that the class of limits for $k_{n}R_{\mu_{n}}$
is precisely the class of all infinitely divisible $R$-transforms.
In classical probability, this is the content of Khintchine's result
on the infinite divisibility in terms of Fourier transform.
\begin{thm}
\emph{(Khintchine Type Characterization for Infinite Divisibility)}
Let $\nu$ be a probability measure on $\mathbb{R}^{2}$, and let
$R_{\nu}$ be its $R$-transform defined on a domain $\Omega$. The
law $\nu$ is bi-freely infinitely divisible if and only if there
exist probability measures $\mu_{n}$ on $\mathbb{R}^{2}$ and unbounded
positive integers $k_{n}$ such that $[\mu_{n}^{(j)}]^{\boxplus k_{n}}\Rightarrow\nu^{(j)}$
($j=1,2$) and $R_{\nu}=\lim_{n\rightarrow\infty}k_{n}R_{\mu_{n}}$
in $\Omega$. \end{thm}
\begin{proof}
We have seen an explanation for the ``only if'' statement at the
beginning of this section. As for the ``if'' part, we have to show
that the limit $R_{\nu}/m=\lim_{n\rightarrow\infty}[k_{n}/m]R_{\mu_{n}}$
is a bi-free partial $R$-transform for each $m\geq2$. This, however, is precisely the content of Theorem 3.6.
\end{proof}
Therefore, Theorem 3.9 and Corollary 3.3 together provide a complete answer to Problem 3.1.

On the other hand, the proof of Theorem 3.6 shows that the infinitely divisible laws can be approximated by Poisson distributions. This perspective leads to another characterization of infinite divisibility in bi-free probability. To describe this result, we need  a preparatory lemma.
\begin{lem} \label{unique2}
Let $\tau$ and $\tau^{\prime}$ be two finite positive Borel measures
on $\mathbb{R}^2$ satisfying
\begin{equation}\label{relation}\frac{t^2}{1+t^{2}}\,d\tau(s,t)=\frac{t^2}{1+t^{2}}\,
d\tau^{\prime}(s,t)\end{equation} on the Borel $\sigma$-field of $\mathbb{R}^2$ and having the same marginal
\[\tau\circ\pi_1^{-1}=\tau^{\prime}\circ\pi_1^{-1}\] with respect to
the projection $\pi_1(s,t)=s$. Then we have  $\tau=\tau^{\prime}$.
\end{lem}
\begin{proof}
It suffices to show that $\tau(E)=\tau^{\prime}(E)$ for any open
rectangle $E=I\times(a,b)$, where $I$ and $(a,b)$ are (bounded or
unbounded) open intervals in $\mathbb{R}$. If $a>0$ or $b<0$, then
the relation (\ref{relation}) shows that the desired identity holds.
If $a=0$, the continuity of $\tau$ and $\tau^{\prime}$ yields that
$\tau(E)=\lim_{\epsilon\to0^+}\tau(I\times(\epsilon,b))
=\lim_{\epsilon\to0^+}\tau^{\prime}(I\times(\epsilon,b))=\tau^{\prime}(E)$.
Similarly, the desired identity holds if $b=0$. On the other hand,
the assumption of admitting the same marginal yields that
$\tau((a,b)\times\mathbb{R})=\tau^{\prime}((a,b)\times\mathbb{R})$, from
which, along with the established result for the cases when
$(a,b)=(0,\infty)$ and $(a,b)=(-\infty,0)$, we see that
$\tau((a,b)\times\{0\})=\tau^{\prime}((a,b)\times\{0\})$. In general, if
$c<0<d$, then applying the above result to the sets
$(a,b)\times(0,d)$, $(a,b)\times\{0\}$, and $(a,b)\times(c,0)$ shows
the desired result. This finishes the proof.
\end{proof}

As seen in the proof of Theorem 3.6, the Poisson approximation $\nu_{k_n,\mu_n}\Rightarrow \nu$ holds along a subsequence of positive integers, thanks to the validity of the weak convergences \begin{equation}\label{eq:4.13}
\frac{k_{n}s^2}{1+s^{2}}\, d\mu_{n}(s,t)\Rightarrow \tau_1,\quad \frac{k_{n}t^2}{1+t^{2}}\, d\mu_{n}(s,t)\Rightarrow \tau_2,\end{equation} and \[\frac{k_{n}st}{\sqrt{1+s^{2}}\sqrt{1+t^{2}}}\, d\mu_{n}(s,t)\Rightarrow \rho\] along the same subsequence. (Of course, the last weak convergence is in fact a genuine limit rather than a subsequential one.)

It is easy to verify that the weak limits $\tau_1$ and $\rho$ must satisfy the identity \[
\frac{t}{\sqrt{1+t^{2}}}\, d\tau_{1}(s,t)=\frac{s}{\sqrt{1+s^{2}}}\, d\rho(s,t)\] on $\mathbb{R}^2$. In particular, this shows that the relationship (3.6) must hold for any weak limit points $\tau_1$ and $\tau_1^{\prime}$ of the sequence $k_{n}s^2/(1+s^{2})\, d\mu_{n}$. Also, the measures $\tau_1$ and $\tau_1^{\prime}$ have the same marginal law on the $s$-axis by the weak convergence \eqref{eq:3.2}. Thus, Lemma 3.10 implies that $\tau_1=\tau_{1}^{\prime}$. Similarly, it can be shown that the limit law $\tau_2$ is also unique. Therefore, the weak convergences \eqref{eq:4.13}, as well as the Poisson approximation to $\nu$, actually hold without passing to subsequences. Thus, we have obtained the following characterization for bi-free infinite divisibility.
\begin{prop}\emph{(Convergence of the accompanying Poisson laws)} Let $\nu$ be a probability measure on $\mathbb{R}^2$. Then $\nu$ is bi-freely infinitely divisible if and only if
there exist probability laws $\mu_n$ on $\mathbb{R}^2$ and unbounded positive integers $k_n$ such that the bi-free compound Poisson laws $\nu_{k_n,\mu_n}$ converge weakly to the measure $\nu$ on $\mathbb{R}^2$.\end{prop}

We would like to conclude this section by pointing out some interesting
properties of $\mathcal{BID}$, the class of bi-freely infinitely
divisible laws. Recall that the integral form of the $R$-transform
of a measure $\nu\in\mathcal{BID}$ is given by
\[
R_{\nu}(z,w)=zR_{\nu_{\boxplus}^{\gamma_{1},\sigma_{1}}}(z)+wR_{\nu_{\boxplus}^{\gamma_{2},\sigma_{2}}}(w)+G_{_{\sqrt{1+s^{2}}\sqrt{1+t^{2}}\, d\rho}}(1/z,1/w),
\]
where $\nu_{\boxplus}^{\gamma_{j},\sigma_{j}}$ ($j=1,2$) and $\rho$
are the limit laws as in (\ref{eq:3.2}) and Theorem 3.2 (3). The
quintuple $\Lambda(\nu)=(\gamma_{1},\gamma_{2},\sigma_{1},\sigma_{2},\rho)$
of limits is uniquely associated with the given infinitely divisible
law $\nu$, regardless of what approximants $\{\mu_{n}\}_{n=1}^{\infty}$
and $\{k_{n}\}_{n=1}^{\infty}$ may be used to obtain them. Being limits,
such quintuples are closed under componentwise addition and multiplication
by positive real numbers. This implies that the bi-free convolution
$\boxplus\boxplus$ can be extended from compactly supported probabilities
to bi-freely infinitely divisible laws; namely, for any $\nu_{1},\nu_{2}\in\mathcal{BID}$,
their bi-free convolution $\nu_{1}\boxplus\boxplus\nu_{2}$ can be defined
as the unique bi-freely infinitely divisible law satisfying
\[
\Lambda(\nu_{1}\boxplus\boxplus\nu_{2})=\Lambda(\nu_{1})+\Lambda(\nu_{2}).
\] We record this finding formally as
\begin{prop}
\emph{(Generalized bi-free convolution)} There exists an associative and commutative binary operation $\boxplus\boxplus: \mathcal{BID}\times  \mathcal{BID} \rightarrow  \mathcal{BID}$ such that for any $\nu_{1},\nu_{2}\in\mathcal{BID}$, the relationship \[R_{\nu_{1}\boxplus\boxplus\nu_{2}}=R_{\nu_1}+R_{\nu_2}\] holds
in $(\mathbb{C}\setminus\mathbb{R})^{2}$.
\end{prop}

In addition, the map $\Lambda$ is injective and weakly continuous. The latter continuity means that if $\Lambda(\nu_{n})=(\gamma_{1n},\gamma_{2n},\sigma_{1n},\sigma_{2n},\rho_n)$, $\Lambda(\nu)=(\gamma_{1},\gamma_{2},\sigma_{1},\sigma_{2},\rho)$, and $\nu_{n}\Rightarrow \nu$, then we have $\gamma_{jn}\rightarrow \gamma_{j}$, $\sigma_{jn}\Rightarrow \sigma_{j}$, and $\rho_n\Rightarrow \rho$ as $n\rightarrow \infty$ for $j=1,2$. This can be easily verified using the free harmonic analysis results in \cite{BerVoicu93} and Proposition 2.3. 

However, the map $\Lambda$
is not surjective, and its actual range will be studied
in the next section. 

Given a law $\nu\in\mathcal{BID}$, Theorem 3.6 shows that for each
$t>0$, there exists a unique law $\nu_{t}\in\mathcal{BID}$ such
that
\[
R_{\nu_{t}}=\lim_{n\rightarrow\infty}[tk_{n}]R_{\mu_{n}}=t\, R_{\nu}
\]
in $(\mathbb{C}\setminus\mathbb{R})^{2}$, where $R_{\nu}=\lim_{n\rightarrow\infty}k_{n}R_{\mu_{n}}$
for some $\{\mu_{n}\}_{n=1}^{\infty}$ and $\{k_{n}\}_{n=1}^{\infty}$.
As we mentioned earlier, this construction of $\nu_{t}$ is independent
of the choice of the sequences $\{\mu_{n}\}_{n=1}^{\infty}$ and $\{k_{n}\}_{n=1}^{\infty}$.
Thus, setting $\nu_{0}=\delta_{(0,0)}$, we then have
\[
\Lambda(\nu_{t})=t\,\Lambda(\nu),\qquad t\geq0,
\]
and the resulting family $\{\nu_{t}\}_{t\geq0}$ forms a weakly continuous
semigroup of probabilities on $\mathbb{R}^{2}$ under the (generalized)
bi-free convolution $\boxplus\boxplus$ in the sense that
\[
\nu_{s+t}=\nu_{s}\boxplus\boxplus\nu_{t},\qquad s,t\geq0,
\]
and $\nu_{t}\Rightarrow \nu_{0}=\delta_{(0,0)}$ as $t\rightarrow 0^{+}$. Conversely, if $ $$\{\nu_{t}\}_{t\geq0}$ is a given weakly continuous
$\boxplus\boxplus$-semigroup of probabilities on $\mathbb{R}^{2}$, then
each $\nu_{t}$ is bi-freely infinitely divisible. In particular,
the law $\nu_{1}$ belongs to the class $\mathcal{BID}$ and generates
the entire process $\{\nu_{t}\}_{t\geq0}$ in terms of bi-free $R$-transform.
Since this connection between infinitely divisible laws and continuous
semigroups will play a role in the next section, we summarize our
discussions into the following
\begin{prop}
\emph{(Embedding Property)} Let $\nu$ be a probability measure on $\mathbb{R}^{2}$.
Then $\nu\in\mathcal{BID}$ if and only if there exists a weakly continuous
$\boxplus\boxplus$-semigroup $\{\nu_{t}\}_{t\geq0}$ of probability measures on $\mathbb{R}^{2}$ satisfying $\nu_{1}=\nu$ and $R_{\nu_{t}}=t\, R_{\nu}$
in $(\mathbb{C}\setminus\mathbb{R})^{2}$.
\end{prop}
Finally, note that the marginals $\{\nu^{(j)}_{t}\}_{t\geq0}$ of a $\boxplus\boxplus$-semigroup $\{\nu_{t}\}_{t\geq0}$ are themselves a continuous semigroup relative to free convolution on $\mathbb{R}$, and we have \[R_{\nu^{(j)}_{t}}(z)=t\,R_{\nu^{(j)}}(z),\quad z\notin\mathbb{R},\] for their one-dimensional $R$-transforms. We recall from \cite{BerVoicu93} that the dynamics of the process $\{\nu^{(j)}_{t}\}_{t\geq0}$ is governed by the complex Burgers' type PDE:\[\partial_{t}G_{\nu^{(j)}_{t}}(z)+R_{\nu^{(j)}_{t}}\left(G_{\nu^{(j)}_{t}}(z)\right)\partial_{z}G_{\nu^{(j)}_{t}}(z)=0\] in the upper half-plane for $t\in[0,\infty)$, where the time-derivative of $G_{\nu^{(j)}_{t}}(z)$ at $t=0$ is equal to the one-dimensional $R$-transform $R_{\nu^{(j)}}(z)$.

\section{bi-free convolution semigroups and L\'{e}vy-Khintchine representations}
In classical probability, finding the L\'{e}vy-Khintchine representation
of a continuous convolution semigroup $\{\mu_{t}\}_{t\geq0}$ of probabilities
on $\mathbb{R}$ is related to the study of its generating functional
$A=\lim_{t\rightarrow0^{+}}(\mu_{t}-\delta_{0})/t$, which is defined
at least on the dual group of $(\mathbb{R},+)$ via
the Fourier transform $\widehat{\mu_{t}}=\exp(tA)$. Fix a law $\nu\in\mathcal{BID}$
and consider the $\boxplus\boxplus$-semigroup $\{\nu_{t}\}_{t\geq0}$
generated by $\nu$. In what follows, we shall study the bi-free $R$-transform
$R_{\nu}$ from an infinitesimal perspective and show that the map
$R_{\nu}$ arises as the time-derivative of the Cauchy transform
$G_{\nu_{t}}$ at time $t=0$. This approach gives rise to a canonical
integral representation for infinitely divisible $R$-transforms.

We start by assuming that $\nu$ is compactly supported, so that its $R$-transform $R_{\nu}$ can be written as an absolutely convergent power series with real coefficients \begin{equation}\label{eq:4.1}R_{\nu}(z,w)=\sum_{m,n \geq 0} \kappa_{m,n}\,z^{m}w^{n}\end{equation} in some bidisk $\Omega=\{(z,w)\in \mathbb{C}^2: |z|, |w|<r\}$. Note that $\kappa_{0,0}=0$, $zR_{\nu^{(1)}}(z)=R_{\nu}(z,0)$, and $wR_{\nu^{(2)}}(w)=R_{\nu}(0,w)$ by Lemma 2.4. So, both $R_{\nu^{(1)}}$ and $R_{\nu^{(2)}}$ also have power series expansions of their own. We conclude that the functions $R_{\nu^{(j)}}$ ($j=1,2$) and $R_{\nu}$ are all uniformly bounded in $\Omega$, and hence the semigroup property shows further that $R_{\nu_{t}^{(j)}}$ ($j=1,2$) and $R_{\nu_{t}}$ all tend to zero uniformly in $\Omega$ as the time parameter $t\rightarrow 0^{+}$. This fact and the definition \eqref{eq:2.3} of $R$-transform imply that there exists a cutoff constant $t_0 > 0$ such that the following identity
\begin{equation} \label{eq:4.2}
G\left(t,K_t^{(1)}(z),K_t^{(2)}(w)\right)
=\frac{zw}{1+tzR_1^{(1)}(z)+twR_1^{(2)}(w)- tR_\mu(z,w)}
\end{equation} holds for $0\leq t\leq t_0$ and $(z,w)\in \Omega^{*}=\Omega \setminus \{(0,0)\}$, using the notations \[G(t,z,w)=G_{\nu_{t}}(z,w), \quad K_t^{(1)}(z)=tR_{\nu^{(1)}}(z)+1/z, \quad K_t^{(2)}(w)=tR_{\nu^{(2)}}(w)+1/w.\] The formula \eqref{eq:4.2} shows that the map $G\left(t,K_t^{(1)}(z),K_t^{(2)}(w)\right)$ is a $C^{1}$-function in $t$ and is holomorphic in $(z,w)$ on the open set $(0,t_0)\times \Omega^{*}$.

Therefore, differentiating \eqref{eq:4.2} at any $t>0$ in the domain $\Omega^{*}$ yields
\begin{eqnarray}\label{eq:4.3}
\partial_{t} G\left(t,K_t^{(1)}(z),K_t^{(2)}(w)\right) & = & -R_{\nu^{(1)}}(z)\partial_{z}
G\left(t,K_t^{(1)}(z),K_t^{(2)}(w)\right)\\ \nonumber & & -R_{\nu^{(2)}}(w)\partial_{w}
G\left(t,K_t^{(1)}(z),K_t^{(2)}(w)\right)\\ \nonumber
 & & +\frac{zw\left(R_\nu(z,w)-zR_{\nu^{(1)}}(z)-wR_{\nu^{(2)}}(w)\right)}{
\left(1+tzR_{\nu^{(1)}}(z)+twR_{\nu^{(2)}}(w)-tR_\nu(z,w)\right)^2}.\end{eqnarray}
Meanwhile, the right-continuity $\nu_t \Rightarrow \delta_{(0,0)}$ and Proposition 2.2 imply that \begin{eqnarray*}
\lim_{t\rightarrow 0^{+}}\partial_{z}
G\left(t,K_t^{(1)}(z),K_t^{(2)}(w)\right) & = & \lim_{t\rightarrow 0^{+}}\int_{\mathbb{R}^{2}}\frac{-1}{(K_t^{(1)}(z)-x)^2(K_t^{(2)}(w)-y)}\,d\nu_t(x,y)\\ & = & -z^2w \end{eqnarray*} and \[\lim_{t\rightarrow 0^{+}}\partial_{w}
G\left(t,K_t^{(1)}(z),K_t^{(2)}(w)\right)=-zw^2.\] Combining these facts with an application of the mean value theorem, we find that for any fixed $(z,w)\in \Omega^{*}$ the function \[t \mapsto G\left(t,K_t^{(1)}(z),K_t^{(2)}(w)\right)\] is also right-differentiable at $t=0$ and this derivative can be further evaluated by taking $t\rightarrow 0^{+}$ in \eqref{eq:4.3}. Taking the fact $\lim_{t\rightarrow 0^{+}}K_t^{(j)}(\lambda)=1/\lambda$ ($j=1,2$) into account, we then obtain the following result.
\begin{prop}We have the right-derivative \[\lim_{t\rightarrow 0^{+}}\frac{G(t,1/z,1/w)-G(0,1/z,1/w)}{t}=zwR_{\nu}(z,w)\] for $(z,w)\in\Omega^{*}$.
\end{prop}
Thus, the distributional derivative $\lim_{t\rightarrow 0^{+}}(\nu_t-\nu_0)/t$ exists at the level of Cauchy transforms. To further analyze this derivative, we will require the tightness of the process $\nu_t$ in finite time. For this purpose, recall from \cite{GHM} that the infinitely divisible measure $\nu$ can be realized as the joint distribution of the \emph{commuting} bounded selfadjoint operators \[a=\ell(f)+\ell(f)^{*}+\Lambda_{\text{left}}(T_1)+\kappa_{1,0}I \quad \text{and} \quad b=r(g)+r(g)^{*}+\Lambda_{\text{right}}(T_2)+\kappa_{0,1}I\] on a certain full Fock space $\mathcal{F}\left(H\right)$ (with the vacuum state) by distinguishing the left action of the creation operator $\ell(f)$ and the gauge operator $\Lambda_{\text{left}}$ from the right action of the operators $r(g)$ and $\Lambda_{\text{right}}$ of the same nature. Here, the vectors $f,g$ in the Hilbert space $H$ and the commuting selfadjoint bounded operators $T_1$ and $T_2$ on $H$ are chosen according to the distribution $\nu$. Consider next the Hilbert space tensor product $L^{2}([0,\infty),dx) \otimes H$, then the realization of the  process $\nu_t$ is given by the two-faced pair $(a_t,b_t)$ of the form: \[a_t=\ell(\chi_{t} \otimes f)+\ell(\chi_{t} \otimes f)^{*}+\Lambda_{\text{left}}(M_t \otimes T_1)+\kappa_{1,0}I\] and \[b_t=r(\chi_{t} \otimes g)+r(\chi_{t} \otimes g)^{*}+\Lambda_{\text{right}}(M_t \otimes T_2)+\kappa_{0,1}I,\] where $\chi_{t}$ is the indicator function of the interval $[0,t)$ and $M_t$ denotes the multiplication operator associated with the function $\chi_{t}$ on $L^{2}([0,\infty),dx)$. We refer the reader to \cite{GHM} for the details of this construction; our point here, however, is that since the norms of $a_t$ and $b_t$ (and hence their spectral radii) are uniformly bounded when $0\leq t\leq t_0$, we are able to conclude that the support of the process $\nu_t$ is uniformly bounded before the cutoff time $t_0$ (actually, within any finite time). In particular, the family $\{\nu_t:0\leq t\leq t_0\}$ is tight. We also mention that the real numbers $\kappa_{1,0}$ and $\kappa_{0,1}$ here are the first two coefficients in the power series expansion \eqref{eq:4.1}, and they represent the mean vector of $\nu$.

We are now ready to present the bi-free L\'{e}vy-Khintchine
representation for compactly supported, infinitely divisible laws.
\begin{thm} \label{BFIDV1} \emph{(Bi-free L\'{e}vy-Khintchine formula for compactly supported measures)}
Let $\nu$ be a compactly supported measure in $\mathcal{BID}$, and let $\{\nu_{t}\}_{t\geq0}$ be the
$\boxplus\boxplus$-semigroup generated by $\nu$. Then there exists a unique triple $(\rho_1,\rho_2,\rho)$ of two compactly supported positive Borel measures $\rho_1$ and $\rho_2$ and a compactly supported Borel signed measure $\rho$ on $\mathbb{R}^2$ such that
\begin{enumerate}[$\qquad(1)$]
\item The weak convergences \[\frac{s^2}{\varepsilon}\,d\nu_{\varepsilon}(s,t)\Rightarrow \rho_1, \quad \frac{t^2}{\varepsilon}\,d\nu_{\varepsilon}(s,t)\Rightarrow \rho_2, \quad \frac{st}{\varepsilon}\,d\nu_{\varepsilon}(s,t)\Rightarrow \rho\] on $\mathbb{R}^2$ and the limits \[\frac{1}{\varepsilon}\int_{\mathbb{R}^2}s\,d\nu_{\varepsilon}(s,t) \rightarrow \kappa_{1,0}
\quad \mathrm{and}\quad \frac{1}{\varepsilon}\int_{\mathbb{R}^2}t\,d\nu_{\varepsilon}(s,t) \rightarrow \kappa_{0,1}\] hold simultaneously as $\epsilon\to0^+$;
\item We have \begin{eqnarray}\label{eq:4.4}
R_{\nu}(z,w) & = & \kappa_{1,0}z+\kappa_{0,1}w+\int_{\mathbb{R}^{2}}\frac{z^{2}}{1-zs}\, d\rho_{1}(s,t)+\int_{\mathbb{R}^{2}}\frac{w^{2}}{1-wt}\, d\rho_{2}(s,t)\\ \nonumber
 &  & +\int_{\mathbb{R}^{2}}\frac{zw}{(1-zs)(1-wt)}\, d\rho(s,t)
\end{eqnarray} for $(z,w)$ in $(\mathbb{C}\setminus\mathbb{R})^2\cup\{(0,0)\}$;
\item The total mass $\rho_{j}(\mathbb{R}^2)$ is equal to the variance of the marginal $\nu^{(j)}$ for $j=1,2$, the number $\rho(\mathbb{R}^2)$ is the covariance of $\nu$, and the measures $\rho_1$, $\rho_2$, and $\rho$ satisfy
\begin{equation}
\begin{cases}
t\, d\rho_{1}=s\, d\rho;\\
s\, d\rho_{2}=t\, d\rho,
\end{cases}\label{eq:4.5}
\end{equation} and
\begin{equation} \label{eq:4.6}
|\rho(\{(0,0)\})|^2 \leq \rho_1(\{(0,0)\})\rho_2(\{(0,0)\}).
\end{equation}
\end{enumerate}
\end{thm}

\begin{proof} First, once we can show that the three families $\{(1/\varepsilon)\,s^2\,d\nu_{\varepsilon}:0<\varepsilon\leq t_0\}$, $\{(1/\varepsilon)\,t^2\,d\nu_{\varepsilon}:0<\varepsilon\leq t_0\}$, and $\{(1/\varepsilon)\,st\,d\nu_{\varepsilon}:0<\varepsilon\leq t_0\}$ are all bounded in total variation norm, the existence for compactly supported limit laws $\rho_1$, $\rho_2$, and $\rho$ will become evident. This is because the family $\{\nu_{\varepsilon}:0\leq\varepsilon\leq t_0\}$ is tight and has a uniformly bounded support. Secondly, after the existence of the limits is established, we shall proceed to prove their uniqueness and the formulas \eqref{eq:4.4}-\eqref{eq:4.6}.

We start with Proposition 4.1 which states that the limit
\begin{align*}
zwR_{\nu}(z,w)&=\lim_{\varepsilon\to0^+}(1/\varepsilon)
\left[G\left(\varepsilon,1/z,1/w\right)-G\left(0,1/z,1/w\right)\right] \\
&=\lim_{\varepsilon\to0^+}\frac{1}{\varepsilon}
\left[\int_{\mathbb{R}^2}\frac{zw}{(1-zs)(1-wt)}\;d\nu_\varepsilon(s,t)-zw\right] \\
&=zw\left[\lim_{\varepsilon\to0^+}\frac{1}{\varepsilon}\int_{\mathbb{R}^2}
\left(\frac{zs}{1-zs}+\frac{wt}{1-wt}+\frac{zwst}{(1-zs)(1-wt)}\right)d\nu_\varepsilon(s,t)\right]
\end{align*} holds for $(z,w)$ in the punctured bidisk $\Omega^{*}$. Hence, we have the identity
\begin{equation}\label{eq:4.7} R_\nu(z,w)=\lim_{\varepsilon\to0^+}\frac{1}{\varepsilon}\int_{\mathbb{R}^2}
\left[\frac{zs}{1-zs}+\frac{wt}{1-wt}+\frac{zwst}{(1-zs)(1-wt)}\right]d\nu_\varepsilon(s,t)\end{equation} in the bidisk $\Omega$, for the integrand is equal to zero for any $(s,t)$ if $z=0=w$.

So, by letting $w=0$ in \eqref{eq:4.7}, Lemma 2.4 implies
\[zR_{\nu^{(1)}}(z)=R_\nu(z,0)
=z\left[\lim_{\varepsilon\to0^+}\frac{1}{\varepsilon}\int_{\mathbb{R}^2}\frac{s}{1-zs}\;d\nu_\epsilon(s,t)\right]\]
or, equivalently,
\begin{equation}\label{eq:4.8}R_{\nu^{(1)}}(z)=
\lim_{\varepsilon\to0^+}\frac{1}{\varepsilon}\int_{\mathbb{R}^2}\frac{s}{1-zs}\;d\nu_\epsilon(s,t)\end{equation} for $|z|<r$. After plugging $z=0$ in this formula, we obtain the limiting formula for the constant $\kappa_{1,0}$. Moreover, by taking the imaginary part of \eqref{eq:4.8}, we reach \[
\lim_{\varepsilon\to0^+}\int_{\mathbb{R}^2}\frac{\Im z}{|1-zs|^2}\frac{s^2}{\varepsilon}\;d\nu_\epsilon(s,t)=\Im R_{\nu^{(1)}}(z),\quad |z|<r.\] This shows that the family $\{(1/\varepsilon)\,s^2\,d\nu_{\varepsilon}:0<\varepsilon\leq t_0\}$ has uniformly bounded total variation norms, because the integrand $\Im z/|1-zs|^2$ is uniformly bounded away from zero and from infinity for $s$ in the uniform support of $\{\nu_{\varepsilon}:0\leq \varepsilon\leq t_0\}$ and for $r/2<|z|<r$, $z\notin \mathbb{R}$. We deduce from the same kind of argument that \begin{equation}\label{eq:4.9}R_{\nu^{(2)}}(w)=
\lim_{\varepsilon\to0^+}\frac{1}{\varepsilon}\int_{\mathbb{R}^2}\frac{t}{1-wt}\;d\nu_\epsilon(s,t),\quad |w|<r,\end{equation} \[\kappa_{0,1}=\lim_{\varepsilon \rightarrow 0^{+}}\frac{1}{\varepsilon}\int_{\mathbb{R}^2}t\,d\nu_{\varepsilon}(s,t),\] and that the set $\{(1/\varepsilon)\,t^2\,d\nu_{\varepsilon}:0<\varepsilon\leq t_0\}$ is bounded in total variation norm. Finally, the boundedness of the remaining family $\{(1/\varepsilon)\,st\,d\nu_{\varepsilon}:0<\varepsilon\leq t_0\}$ is an easy consequence of the Cauchy-Schwarz inequality. So, the existence of the limit laws is proved.

In view of \eqref{eq:4.8} and \eqref{eq:4.9}, we can now split \eqref{eq:4.7} into three limits, replace the first two with the corresponding $R$-transforms, and finally get \begin{equation}\label{eq:4.10}\lim_{\varepsilon\to0^+}\int_{\mathbb{R}^2}
\frac{zw}{(1-zs)(1-wt)}\frac{st}{\varepsilon}\,d\nu_\varepsilon(s,t)=R_\nu(z,w)-zR_{\nu^{(1)}}(z)-wR_{\nu^{(2)}}(w)\end{equation} for $(z,w)\in\Omega$. This shows that if $\rho$ and $\rho^{\prime}$ are two weak limits of the signed measures $\{(1/\varepsilon)\,st\,d\nu_{\varepsilon}:0<\varepsilon\leq t_0\}$ as $\varepsilon \rightarrow 0^{+}$, then they must satisfy \[G_{\rho}(1/z,1/w)=G_{\rho^{\prime}}(1/z,1/w)\] for $(z,w)$ in the open set $\Omega^{*}$ (and hence everywhere by analytic extension), proving that $\rho=\rho^{\prime}$. Therefore, the limit law $\rho$ is unique and $(1/\varepsilon)\,st\,d\nu_{\varepsilon} \Rightarrow \rho$. In addition, it follows from \eqref{eq:4.10} and the power series expansion \eqref{eq:4.1} that \[\rho(\mathbb{R}^2)=\kappa_{1,1}= \text{Covariance}(\nu).\]

To see the uniqueness of the other two limit laws $\rho_1$ and $\rho_2$, we first expand the integrand in \eqref{eq:4.8} into a power series of $z$ and then use the fact that any limit $\rho_1$ of $\{(1/\varepsilon)\,s^2\,d\nu_{\varepsilon}:0<\varepsilon\leq t_0\}$ is compactly supported to get
\begin{equation} \label{eq:4.11}
R_{\nu^{(1)}}(z)=\kappa_{1,0}+\sum_{m=0}^\infty
\left(\int_{\mathbb{R}^2}s^m\;d\rho_1(s,t)\right)z^{m+1},\quad |z|<r.
\end{equation} Similarly, for $|w|<r$, we have
\begin{equation} \label{eq:4.12}
R_{\nu^{(2)}}(w)=\kappa_{0,1}+\sum_{n=0}^\infty
\left(\int_{\mathbb{R}^2}t^n\;d\rho_2(s,t)\right)w^{n+1}.
\end{equation} On the other hand, we should do the same for \eqref{eq:4.10}; only this time we will use the following decomposition
\[\frac{st}{(1-zs)(1-wt)}=st+\frac{ws}{1-wt}\cdot
t^2+\frac{zt}{(1-zs)(1-wt)}\cdot s^2\] to obtain
\begin{eqnarray*}
R_\nu(z,w)-zR_{\nu^{(1)}}(z)-wR_{\nu^{(2)}}(w) & = &\kappa_{1,1}zw+\sum_{n\geq0}\left(\int_{\mathbb{R}^2}st^n\;d\rho_2(s,t)\right)zw^{n+2} \\ & & +
\sum_{m,n\geq0}\left(\int_{\mathbb{R}^2}s^mt^{n+1}\;d\rho_1(s,t)\right)z^{m+2}w^{n+1}.
\end{eqnarray*} Combining this with \eqref{eq:4.11}, \eqref{eq:4.12}, and the original power series expansion \eqref{eq:4.1} of $R_{\nu}$, we have shown the following identity
\begin{eqnarray*}
\sum_{m,n \geq 0} \kappa_{m,n}\,z^{m}w^{n} & = & \kappa_{1,0}z+\kappa_{0,1}w+\sum_{m=0}^\infty
M_{m,0}(\rho_{1})z^{m+2}+\sum_{n=0}^\infty
M_{0,n}(\rho_{2})w^{n+2}\\ & &+\kappa_{1,1}zw+\sum_{n\geq0}M_{1,n}(\rho_{2})zw^{n+2}+
\sum_{m,n\geq0}M_{m,n+1}(\rho_{1})z^{m+2}w^{n+1}
\end{eqnarray*} of power series in the open set $\Omega$, where the notation \[M_{m,n}(\rho_{j})=\int_{\mathbb{R}^2}s^m\,t^n\;d\rho_j(s,t),\quad m,n\geq 0,\quad j=1,2.\] Since all these power series converge absolutely, we are allowed to rearrange the order of summation freely. By the uniqueness of power series expansion in open sets, we conclude that the moments of the limiting measure $\rho_1$ (and hence $\rho_1$ itself) are uniquely determined by the given sequence $\{\kappa_{m,n}\}_{m,n\geq 0}$ of coefficients and therefore the weak convergence $(1/\varepsilon)\,s^2\,d\nu_{\varepsilon} \Rightarrow \rho_1$ holds. Also, by comparing the coefficients in the preceding identity of power series, we have \[\rho_{1}(\mathbb{R}^2)= \kappa_{2,0}=\text{Variance}(\nu^{(1)}).\] The uniqueness of $\rho_2$ and its statistics can be shown in the same way, and we shall not repeat this argument. Thus, the statement (1) of the theorem is proved.

The integral representation \eqref{eq:4.4} is a direct consequence of \eqref{eq:4.7} and the convergences in (1).

To show the system \eqref{eq:4.5}, take any continuous and bounded function $\varphi$ on $\mathbb{R}^2$, the weak convergence results in (1) imply \[\int_{\mathbb{R}^2}\varphi(s,t)\, t\;d\rho_1(s,t)=\lim_{\varepsilon\to0^+}(1/\varepsilon)\int_{\mathbb{R}^2}\varphi(s,t)\,t\,s^2\;
d\nu_\varepsilon(s,t)=\int_{\mathbb{R}^2}\varphi(s,t)\, s\;d\rho(s,t),\] from
which we deduce that $t\,d\rho_1(s,t)=s\,d\rho(s,t)$. Similarly,
$s\,d\rho_2(s,t)=t\,d\rho(s,t)$ holds.

Finally, for the inequality \eqref{eq:4.6}, let $\varphi_n$ be a sequence of continuous functions such that $0\leq \varphi_n(s,t) \leq 1$ and $\lim_{n\rightarrow \infty}\varphi_n(s,t)=I_{\{(0,0)\}}(s,t)$, the indicator function of the singleton $\{(0,0)\}$, for all $(s,t)\in \mathbb{R}^2$. An application of the Cauchy-Schwarz inequality to the measure $(1/\varepsilon)\varphi_n\,d\nu_{\varepsilon}$ shows that \[\left|\int_{\mathbb{R}^2}\varphi_n(s,t)\frac{st}{\varepsilon}\;d\nu_\varepsilon(s,t)\right|^2\leq
\int_{\mathbb{R}^2}\varphi_n(s,t)\frac{s^2}{\varepsilon}\;d\nu_\varepsilon(s,t)
\int_{\mathbb{R}^2}\varphi_n(s,t)\frac{t^2}{\varepsilon}\;d\nu_\varepsilon(s,t)\] for all $\varepsilon>0$. By letting $\varepsilon\rightarrow 0^{+}$, we obtain the estimate \[\left|\int_{\mathbb{R}^2}\varphi_n(s,t)\;d\rho(s,t)\right|^2\leq
\int_{\mathbb{R}^2}\varphi_n(s,t)\;d\rho_1(s,t)
\int_{\mathbb{R}^2}\varphi_n(s,t)\;d\rho_2(s,t),\quad n\geq 1.\] The inequality \eqref{eq:4.6} follows from this estimate and the dominated convergence theorem. The theorem is now completely proved.
\end{proof}
We mention that the integral formula \eqref{eq:4.4} for compactly supported, bi-freely infinitely divisible measures was first obtained in \cite{GHM} by means of combinatorial and operator-theoretical methods.

We next present our main result in which the boundedness condition for the support of $\nu$ is no longer needed. Note that the following result is stronger than the one obtained in \cite{GHM}, because it provides a complete parametrization for the entire class $\mathcal{BID}$.
\begin{thm} \label{BFIDV3} \emph{(General bi-free L\'{e}vy-Khintchine
representation)}
Let $R$ be a given holomorphic function defined on the product domain $\Omega=(\Delta\cup\overline{\Delta})\times(\Delta\cup\overline{\Delta})$ associated with some Stolz angle $\Delta$. Then
the following statements are equivalent:
\begin{enumerate} [$\qquad(1)$]
\item {There exists a law $\nu\in \mathcal{BID}$ such that $R=R_{\nu}$ on $\Omega$.}
\item {There exist $\gamma_1,\gamma_2\in \mathbb{R}$, two finite Borel positive measures $\rho_1$ and $\rho_2$ on $\mathbb{R}^2$, and a finite Borel signed measure $\rho$ on $\mathbb{R}^2$ such that
\begin{equation}
\begin{cases}
t/\sqrt{1+t^2}\, d\rho_{1}=s/\sqrt{1+s^2}\, d\rho;\\
s/\sqrt{1+s^2}\, d\rho_{2}=t/\sqrt{1+t^2}\, d\rho;\\
|\rho(\{(0,0)\})|^2 \leq \rho_1(\{(0,0)\})\rho_2(\{(0,0)\}),
\end{cases}\label{eq:4.15}
\end{equation} and the function $R$ extends analytically to $(\mathbb{C}\setminus \mathbb{R})^2$ via the formula: \begin{eqnarray}\label{eq:4.16}
R(z,w) & = & \gamma_{1}z+\gamma_{2}w+\int_{\mathbb{R}^{2}}\frac{z^{2}+zs}{1-zs}\, d\rho_{1}(s,t)+\int_{\mathbb{R}^{2}}\frac{w^{2}+wt}{1-wt}\, d\rho_{2}(s,t)\\
 &  & +\int_{\mathbb{R}^{2}}\frac{zw\sqrt{1+s^{2}}\sqrt{1+t^{2}}}{(1-zs)(1-wt)}\, d\rho(s,t). \nonumber
\end{eqnarray}}
\end{enumerate}
In this case, the quintuple $(\gamma_1, \gamma_2, \rho_{1}, \rho_{2}, \rho)$ is unique, and we have the marginal law $\nu^{(j)}=\nu_{\boxplus}^{\gamma_j,\rho_{j}^{(j)}}$ for $j=1,2$.
\end{thm}
\begin{proof} It is clear that only the implication from (2) to (1) needs a proof.  Suppose we are given the integral form \eqref{eq:4.16}, whose representing measures $\rho_1$, $\rho_2$, and $\rho$ satisfy the system \eqref{eq:4.15}. Let $S=\{(s,0)\in \mathbb{R}^2: s\neq 0\}$ and $T=\{(0,t)\in \mathbb{R}^2: t\neq 0\}$ be the two punctured coordinate axes on the plane and let $U=\mathbb{R}^2 \setminus (S\cup T \cup \{(0,0)\})$ be the slit plane. The sets $S$, $T$, and $U$ are Borel measurable, and we can consider the following decompositions \begin{equation*}
\begin{cases}
\rho_j=\rho_j(\{(0,0)\})\delta_{(0,0)}+\rho_{j}^{S}+\rho_{j}^{T}+\rho_{j}^{U},\quad j=1,2,\\
\rho=\rho(\{(0,0)\})\delta_{(0,0)}+\rho^{S}+\rho^{T}+\rho^{U}.
\end{cases}
\end{equation*} A notation like $\rho^{S}$ here means the restriction of the measure $\rho$ on the Borel set $S$, i.e., \[\rho^{S}(E)=\rho(E\cap S)\] for all Borel sets $E\subset \mathbb{R}^2$. We shall identify the measures restricted on $S$ or $T$ as measures on $\mathbb{R}$ and write, with a slight abuse of notation, that $\rho_{j}^{S}=d\rho_{j}^{S}(s)$ and $\rho_{j}^{T}=d\rho_{j}^{T}(t)$.

To each $n\geq 1$, we introduce the set $T_{n}=\{(0,t)\in \mathbb{R}^2: |t|\geq 1/n\}$ and observe from the dominated convergence theorem and the system \eqref{eq:4.15} that  \begin{eqnarray*}
\rho_1(T) & = &\lim_{n\rightarrow \infty} \rho_1(T_n)=\lim_{n\rightarrow \infty}\int_{T_n} \frac{\sqrt{1+t^2}}{t}\frac{t}{\sqrt{1+t^2}}\,d\rho_1(s,t)\\
 &  & = \lim_{n\rightarrow \infty}\int_{T_n} \frac{\sqrt{1+t^2}}{t}\frac{s}{\sqrt{1+s^2}}\,d\rho(s,t)=\lim_{n\rightarrow \infty}\int_{T_n} 0\,d\rho(s,t)=0.
\end{eqnarray*} This implies that the restricted measure $\rho_1^{T}$ is in fact the zero measure. Similarly, one can show that the measures $\rho_2^{S}$, $\rho^{S}$, and $\rho^{T}$ are also equal to the zero measure. Accordingly, the integral form \eqref{eq:4.16} can be decomposed into
\begin{eqnarray*}
R(z,w) & = & \rho_1(\{(0,0)\})z^2+\rho_2(\{(0,0)\})w^2+\rho(\{(0,0)\})zw\\ & & +z\left[\gamma_{1}+\int_{\mathbb{R}\setminus \{0\}}\frac{z+s}{1-zs}\, d\rho_{1}^{S}(s)\right]+w\left[\gamma_{2}+\int_{\mathbb{R}\setminus \{0\}}\frac{w+t}{1-wt}\, d\rho_{2}^{T}(t)\right]\\
 &  & +\int_{U}\frac{z^2+zs}{1-zs}\, d\rho_{1}^{U}(s,t)+\int_{U}\frac{w^2+wt}{1-wt}\, d\rho_{2}^{U}(s,t)\\ & &+\int_{U}\frac{zw\sqrt{1+s^{2}}\sqrt{1+t^{2}}}{(1-zs)(1-wt)}\, d\rho^{U}(s,t), \nonumber
\end{eqnarray*} in which we find that \[\rho_1(\{(0,0)\})z^2+\rho_2(\{(0,0)\})w^2+\rho(\{(0,0)\})zw=R_{\mu_{1}}(z,w)\] for some bi-free Gaussian law $\mu_1$ by \eqref{eq:4.15} and that \[z\left[\gamma_{1}+\int_{\mathbb{R}\setminus \{0\}}\frac{z+s}{1-zs}\, d\rho_{1}^{S}(s)\right]+w\left[\gamma_{2}+\int_{\mathbb{R}\setminus \{0\}}\frac{w+t}{1-wt}\, d\rho_{2}^{T}(t)\right]\] is the bi-free $R$-transform of the product measure \[\mu_2=\nu_{\boxplus}^{\gamma_1, \rho_{1}^{S}}\otimes \nu_{\boxplus}^{\gamma_2, \rho_{2}^{T}}.\] Note that both $\mu_1$ and $\mu_2$ are bi-freely infinitely divisible.

We shall argue that the remaining integral form $R_3=R-R_{\mu_1}-R_{\mu_2}$ is also an infinitely divisible bi-free $R$-transform. Toward this end we consider the truncations \[\rho_{jn}^{U}=\varphi_{n}\,\rho_j^{U}\quad \text{and} \quad \rho_{n}^{U}=\varphi_{n}\,\rho^{U} \quad (j=1,2),\] where $0\leq \varphi_{n} \leq 1$ is a continuous function on $\mathbb{R}^2$ such that $\varphi_{n}(s,t)=1$ for $(s,t)\in U_{n}=\{(s,t)\in \mathbb{R}^2: |s|\geq 1/n, |t|\geq 1/n\}$ and $\varphi_{n}(s,t)=0$ on the complement $\mathbb{R}^2\setminus U_{n+1}$. Clearly, we have $\rho_{jn}^{U}\Rightarrow \rho_{j}^{U}$ ($j=1,2$) and $\rho_{n}^{U}\Rightarrow \rho^{U}$ as $n\rightarrow \infty$, so that the corresponding sequence \begin{eqnarray}\label{eq:4.17}
R_n(z,w) & = & \int_{U_{n+1}}\frac{z^2+zs}{1-zs}\, d\rho_{1n}^{U}(s,t)+\int_{U_{n+1}}\frac{w^2+wt}{1-wt}\, d\rho_{2n}^{U}(s,t) \\ & &+\int_{U_{n+1}}\frac{zw\sqrt{1+s^{2}}\sqrt{1+t^{2}}}{(1-zs)(1-wt)}\, d\rho_{n}^{U}(s,t)\nonumber
\end{eqnarray} tends to $R_3(z,w)$ for each $z,w \notin \mathbb{R}$. In addition, it is easy to see that the limit \[\lim_{y,v\rightarrow 0^{+}}R_n (-iy,-iv)= 0\] holds uniformly for all $n$.

Therefore, if we can show that each $R_n$ is an infinitely divisible bi-free $R$-transform, then Proposition 2.6 and the fact that the family $\mathcal{BID}$ is closed under the topology of weak convergence would imply that $R_3=R_{\mu_3}$ for some $\mu_3\in \mathcal{BID}$. In that way, the desired probability measure $\nu$ could be given by the generalized bi-free convolution \[\nu=\mu_1\boxplus\boxplus\mu_2\boxplus\boxplus\mu_3\] among these infinitely divisible laws. For this purpose, we fix $n$ and introduce a finite Borel signed measure $\tau$ on $U_{n+1}$ by \begin{equation}\label{eq:4.18}\tau=\frac{\sqrt{1+s^2}\sqrt{1+t^2}}{st}\,d\rho_{n}^{U}.\end{equation} The conditions \eqref{eq:4.15} imply that the measure \begin{equation}\label{eq:4.19}\tau=\frac{1+s^2}{s^2}\,d\rho_{1n}^{U}=\frac{1+t^2}{t^2}\,d\rho_{2n}^{U},\end{equation} and hence it is in fact a positive measure on the set $U_{n+1}$. If $\lambda=\tau(U_{n+1})=0$, then \eqref{eq:4.18} and \eqref{eq:4.19} imply that the measures $\rho_{1n}^{U}$, $\rho_{2n}^{U}$, and $\rho_{n}^{U}$ are all equal to the zero measure. So, the function $R_n$ in this case is constantly zero, and hence it is the infinitely divisible $R$-transform corresponding to the point mass at $(0,0)$. Therefore, we assume $\lambda>0$ in the sequel and normalize the measure $\tau$ to get the probability law \[\mu=\tau/\lambda.\]

Finally, define the constants \[a=-\lambda\int_{U_{n+1}}\frac{s}{1+s^2}\, d\mu(s,t)\quad \text{and} \quad b=-\lambda\int_{U_{n+1}}\frac{t}{1+t^2}\, d\mu(s,t).\] We now combine \eqref{eq:4.17}, \eqref{eq:4.18}, and \eqref{eq:4.19} with the identity \[\frac{z+x}{1-zx}\frac{x^2}{1+x^2}=\frac{x}{1-zx}-\frac{x}{1+x^2},\quad z \in \mathbb{C} \setminus \mathbb{R},\,\, x\in \mathbb{R},\]to get \[R_{n}(z,w)=az+bw-\lambda+\lambda\int_{U_{n+1}}\frac{1}{(1-zs)(1-wt)}\, d\mu(s,t).
\] This proves that $R_n$ is the $R$-transform of the bi-free convolution \[\delta_{(a,b)}\boxplus\boxplus \nu_{\lambda,\mu},\] which is indeed infinitely divisible. The proof is now completed.
\end{proof}
The attentive reader may notice that the integral representation
\eqref{eq:4.16} could have been derived directly from the general
limit theorems in Section 3 through a discretization process of the
$\boxplus\boxplus$-semigroup $\{\nu_t\}_{t\geq 0}$. However, the
approach undertaken here has the advantage that not only does it
reveal how the conditions \eqref{eq:4.15} arise naturally from the
dynamical view of $R$-transform (hence justifying the name
"L\'{e}vy-Khintchine formula"), but it also demonstrates that at any
time $t$, the process $\nu_{t}$ can be realized as
\[\nu_t=\text{Gaussian}\boxplus\boxplus\text{1-D infinitely
divisible product}\boxplus\boxplus\text{Poisson limit},\] which
resembles its Fock space model in the bounded support case.

Finally, we remark that the integral formulas \eqref{eq:4.4} and \eqref{eq:4.16} are equivalent when the infinitely divisible law $\nu$ is compactly supported. This follows from an easy substitution: $d\rho^{\prime}_{1}=(1+s^{2})\, d\rho_{1}$,
$d\rho^{\prime}_{2}=(1+t^{2})\, d\rho_{2}$, $d\rho^{\prime}=\sqrt{1+s^{2}}\sqrt{1+t^{2}}\, d\rho$, and
\[
a_{1}=\gamma_{1}+\int_{\mathbb{R}^{2}}s\, d\rho_{1}(s,t),\quad a_{2}=\gamma_{2}+\int_{\mathbb{R}^{2}}t\, d\rho_{2}(s,t),
\] which turns the integral form \eqref{eq:4.16} into
\begin{eqnarray*}
R_{\nu}(z,w) & = & a_{1}z+a_{2}w+\int_{\mathbb{R}^{2}}\frac{z^{2}}{1-zs}\, d\rho_{1}^{\prime}(s,t)+\int_{\mathbb{R}^{2}}\frac{w^{2}}{1-wt}\, d\rho_{2}^{\prime}(s,t)\\
 &  & +\int_{\mathbb{R}^{2}}\frac{zw}{(1-zs)(1-wt)}\, d\rho^{\prime}(s,t),
\end{eqnarray*} 
and the system \eqref{eq:4.15} now becomes 
\[
\begin{cases}
t\, d\rho_{1}^{\prime}=s\, d\rho^{\prime};\\
s\, d\rho_{2}^{\prime}=t\, d\rho^{\prime};\\
|\rho^{\prime}(\{(0,0)\})|^2 \leq \rho_{1}^{\prime}(\{(0,0)\})\rho_{2}^{\prime}(\{(0,0)\}).
\end{cases}\]


\begin{thebibliography}{99}

\bibitem{BerPata99} H. Bercovici, V. Pata, Stable laws and
domains of attraction in free probability theory {\it Annals of
Mathematics} {\bf149} (1999), 1023-1060.

\bibitem{BerVoicu93} H. Bercovici, D. Voiculescu, Free Convolutions of measures with
unbounded support, {\it Indiana Univ. Math. J.} {\bf42} (3) (1993)
733-773.

\bibitem{CNS} I. Charlesworth, B. Nelson, and P. Skoufranis, On two-faced families of non-commutative
random variables,  \textit{Canad. J. Math.} \textbf{67} (2015), no. 6, 1290-1325.

\bibitem{CNS2} ---------, Combinatorics of Bi-Freeness with Amalgamation, \textit{Comm. Math. Phys.} \textbf{338} (2015), no. 2, 801-847.

\bibitem{FW} A. Freslon and M. Weber, On bi-free De Finetti theorems, preprint (2015), arXiv:1501.05124.

\bibitem{GHM} Y. Gu, H.-W. Huang, and J. A. Mingo, An analogue of the L\'{e}vy-Hin\v{c}in formula for
bi-free infinitely divisible distributions, preprint (2014), to appear in \textit{Indiana Univ. Math. J.}, arXiv:1501.05369.

\bibitem{MN} M. Mastnak and A. Nica, Double-ended queues and joint moments of left-right canonical
operators on a full Fock space, \textit{Internat. J. Math.} \textbf{26} (2015), No.2, 34 pp.

\bibitem{Skoufranis14} P. Skoufranis, Independence and partial $R$-transforms in bi-free
probability, preprint (2014), to appear in \textit{Ann. Inst. Henri Poincare Probab. Stat.}, arXiv:1410.4265.

\bibitem{Skoufranis15} ---------, A combinatorial approach to Voiculescu's bi-free partial
transforms, preprint (2015), arXiv:1504.06005.

\bibitem{V14} D. V. Voiculescu, Free Probability for Pairs of Faces I, \textit{Comm. Math. Phys.}
\textbf{332} (2014), 955-980.

\bibitem{V15} ---------, Free Probability for Pairs of Faces II: 2-Variables Bi-free Partial
$R$-Transform and Systems with Rank $\leq 1$ Commutation, \textit{Ann. Inst. Henri Poincare Probab. Stat.} \textbf{52} (2016), no. 1, 1-15.

\bibitem{V15ST} ---------, Free probability for pairs of faces III: $2$-variables bi-free partial
$S$- and $T$-transforms,  \textit{J. Funct. Anal.} \textbf{270} (2016), no. 10, 3623-3638.

\bibitem{V15EXT} ---------, Free probability for pairs of faces IV: bi-free
extremes in the plane, preprint (2015), to appear in \textit{J. Theoret. Probab.}, arXiv:1505.05020.

\end{thebibliography}
\end{document}